\documentclass[mathpazo]{cicp}
\usepackage{caption}
\usepackage{url,lineno,subcaption}
\usepackage{amsmath}
\usepackage{bm}
\usepackage{CJK}
\usepackage{enumerate}
\usepackage{color} 
\usepackage{amssymb}
\usepackage{geometry}
\usepackage{amsfonts}
\usepackage{indentfirst}
\usepackage{mathtools}
\usepackage{setspace}
\usepackage[ruled,linesnumbered]{algorithm2e}
\usepackage{graphicx}

\usepackage{booktabs} 
\usepackage{multirow} 

\usepackage{comment}
	\usepackage{amsthm}

\usepackage{microtype}
\usepackage{hyperref}

\allowdisplaybreaks[4]

\begin{document}

\title{Physics-Informed Neural Networks for Biot's Model via Fixed-Stress Splitting and Energy Natural Gradient Descent}


\author[Sun K X et.~al.]{Kexin Sun\affil{1},
      Qiang Liu\affil{2}, Minfu Feng\affil{1}, and Mingchao Cai\affil{3}\comma\corrauth}
\address{\affilnum{1}\ School of Mathematics,
         Sichuan University,
         Chengdu, Sichuan 610064, P.R. China. \\
          \affilnum{2}\ School of Mathematical Sciences, 
          Shenzhen University, Shenzhen, Guangdong 518055, P.R. China. \\
         \affilnum{3}\ Department of Mathematics, Morgan State University, Baltimore, MD 21251, USA.}
\emails{{\tt sunkexinkk@outlook.com} (K.~Sun), {\tt matliu@szu.edu.cn} (Q.~Liu),
         {\tt fmf@scu.edu.cn} (M.~Feng), {\tt Mingchao.Cai@morgan.edu} (M.~Cai)}

\begin{abstract}
Physics-Informed Neural Networks (PINNs) have recently gained considerable attention as 
a mesh-free framework for solving partial differential equations. Nevertheless, their performance 
deteriorates when applied to strongly coupled multiphysics systems, such as Biot’s consolidation model, 
due to severely ill-conditioned optimization landscapes. In this work, we propose a robust PINN-based solver, 
termed FS-ENGD-PINN, which synergistically integrates physics-based decoupling with geometry-aware optimization. 
Specifically, the Fixed-Stress (FS) splitting scheme is employed to decompose the coupled poroelastic system 
into contractive mechanics and flow subproblems, thereby significantly improving training stability and convergence. 
To further accelerate optimization, we adopt Energy Natural Gradient Descent (ENGD), which approximates the 
Newton direction in function space effectively mitigates stiffness-induced slow convergence. Moreover, 
to address volumetric locking arising in the nearly incompressible regime, we 
incorporate a three-field mixed formulation with an additional total pressure variable into the PINN framework. 
Extensive numerical experiments demonstrate that the proposed FS-ENGD-PINN consistently outperforms 
standard PINN formulations in terms of accuracy and robustness, providing a unified and reliable learning-based 
solver for poroelasticity across a wide range of material parameters.

\end{abstract}


\maketitle
\section{Introduction}
Biot's consolidation model \cite{biot1941general} describes the coupled behavior of fluid flow and mechanics in porous media and has been widely used in various fields such as geomechanics \cite{kim2011stability, lee2016pressure}, petroleum engineering \cite{gudala2020numerical}, materials science \cite{zhang2007biot}, CO$_2$ sequestration \cite{anthony2020simulation}, and biomechanics \cite{ju2020parameter}. 
The model consists of a system of partial differential equations (PDEs) that govern the displacement of the solid matrix and the pressure of the fluid within the pores. 
Due to the strong coupling characteristics of fluid flow and solid deformation and the complex physical mechanism, solving Biot's model is a challenging task. 

Traditional numerical methods, such as the finite difference method (FDM) \cite{hosokawa2005simulation, masson2010finite}, finite volume method (FVM) \cite{nordbotten2016stable, sokolova2019multiscale}, and finite element method (FEM) \cite{kadeethum2020finite, he2025parameter}, have been widely used to solve Biot's model. 
Among them, the FEM has established itself as the dominant computational framework due to its versatility in handling complex geometries and boundary conditions.
Recent advancements in FEM-based approaches aim to enhance accuracy and stability, including mixed finite element methods (MFE) \cite{phillips2007coupling, bean2017block}, discontinuous Galerkin methods (DG) \cite{chen2013analysis, he2024analysis}, weak Galerkin methods (WG) \cite{gu2023weak, wang2024full}, and virtual element methods (VEM) \cite{liu2023virtual, botti2025fully}. 
However, standard formulations often suffer from numerical instabilities, such as pressure oscillations and volumetric locking, particularly near the incompressible limit (i.e., when Poisson's ratio approaches 0.5) \cite{phillips2009overcoming, cai2015comparisons, yi2017study}. 
Furthermore, these mesh-based methods incur high computational costs for mesh generation and refinement, especially in high-dimensional problems or complex geometries.
To alleviate the computational burden of strong coupling, various iterative solution strategies have been developed, such as the undrained-split, the fixed-stress split, the drained split, and the fixed-strain split iterative methods \cite{kim2011stability, kimstability, mikelic2013convergence, mikelic2014numerical}. 
Among them, the fixed-stress splitting method has gained significant attention due to its unconditional stability and convergence properties \cite{both2017robust, bause2017space, storvik2019optimization, aronson2024pressure}, effectively decoupling the system into contractive subproblems for flow and mechanics.

With the development of deep learning techniques, Physics-Informed Neural Networks (PINNs) have emerged as a promising approach for solving PDEs \cite{raissi2017physics, raissi2019physics, cuomo2022scientific}. 
By embedding physical laws directly into the loss function of deep neural networks, PINNs offer compelling advantages over traditional numerical methods, such as mesh-free properties, natural handling of inverse problems, and automatic differentiation.
These capabilities allow PINNs to effectively capture complex behaviors and patterns in the solution space, making them suitable for solving a wide range of problems, including fluid dynamics \cite{cai2021physics, zhao2024comprehensive, ali2025machine}, and solid mechanics \cite{hu2024physics, wang2025novel}.
Following the successful application of PINNs to solid mechanics \cite{haghighat2021physics} and multiphase fluid flow \cite{almajid2022prediction}, recent efforts have expanded to fully coupled poromechanics.

Nevertheless, strongly coupled poromechanics remains challenging for PINNs. 
Biot-type systems exhibit pronounced multi-scale behavior with significant scale disparities, an intrinsic saddle-point structure, and tightly coupled field interactions. 
These features lead to severe ill-conditioning in the learning problem, often hindering standard optimization algorithms from achieving high-accuracy solutions \cite{wang2021understanding, rathore2024challenges}.
To address this complex hydromechanical coupling, recent studies \cite{haghighat2022physics,cai2023combination} introduced the fixed-stress splitting sequential training into the PINN framework, showing better stability than fully-coupled models.
More recently, Millevoi et al. \cite{millevoi2024physics} demonstrated that physical partitioning alone is not always sufficient.
They highlighted that standard optimizers (e.g., Adam and L-BFGS) still struggle with severe ill-conditioning and gradient conflicts, often requiring internal observational data for sensor-driven training to accurately capture sharp physical gradients.

To address these challenges, we propose a divide-and-conquer strategy that alleviates stiffness at both the physical modeling level and the optimization stage. 
At the physical level, we incorporate the fixed-stress splitting iterative scheme to decouple the strongly coupled poroelastic system. 
However, despite this physical decoupling, relying solely on conventional gradient-based optimizers often causes the training to stagnate in local minima due to the highly non-convex and ill-conditioned optimization landscape.
Moreover, the squared-residual formulation commonly used in PDE-based loss functions effectively squares the condition number of the underlying operator, further exacerbating ill-conditioning and slowing convergence of 
iterative optimization methods \cite{wang2021understanding, zeng2022competitive}. 
To overcome these optimization bottlenecks, we introduce Energy Natural Gradient Descent (ENGD) at the optimization stage \cite{muller2023achieving}. 
By employing the Gram matrix of the loss functional as a preconditioner, ENGD yields an update direction in function space that closely approximates the Newton direction, thereby mitigating stiffness and flattening steep ravines in the optimization landscape. 
The synergistic combination of physics-based fixed-stress splitting and geometry-aware ENGD optimization enables our framework to surpass the accuracy limitations of traditional PINNs and achieve high-fidelity solutions with substantially fewer training iterations.

The main contributions of this paper are summarized as follows:
\begin{itemize}
    \item We propose a novel framework, \textbf{FS-ENGD-PINN}, which synergizes the fixed-stress splitting iterative method with energy natural gradient descent to solve Biot's consolidation problem. 
    This two-level strategy effectively solves the stiffness problems arising from physical coupling and numerical optimization, ensuring strong stability and fast convergence.
    \item We provide a comprehensive analysis that links the contraction property of the fixed-stress splitting (at the physical level) with the optimization efficiency of the natural energy gradient descent (at the optimization stage).
    \item We address the critical issue of volumetric locking near the incompressible limit ($\nu\to 0.5$) by developing a \textbf{mixed PINN approach} based on a three-field formulation ($p-\boldsymbol{u}-\xi$). 
    By treating the total pressure $\xi$ as an independent network output, we eliminate pressure oscillations and ensure stability even under extreme parameter settings where standard two-field formulations fail.
    \item We validate our method through extensive numerical benchmarks, including 2D and 3D linear problems with mixed boundary conditions, a problem with complex spatiotemporal solutions, the Mandel problem, and the highly heterogeneous layered Terzaghi's problem. 
    The results demonstrate that FS-ENGD-PINN significantly outperforms baseline methods in terms of both runtime efficiency (time-to-accuracy) and final solution quality, while exhibiting strong robustness in resolving highly heterogeneous media with sharp physical interfaces.
\end{itemize}

The remainder of this paper is organized as follows. 
Section \ref{sec:preliminaries} outlines the mathematical formulation of Biot's model, including the locking-free three-field formulation. 
Section \ref{sec:methods} details the architecture of the FS-ENGD-PINN framework and provides the convergence analysis. 
Section \ref{sec:experiments} presents numerical validation against analytical and numerical benchmarks. 
Finally, Section \ref{sec:conclusion} concludes the work and discusses future research directions.

\section{Mathematical Formulations}\label{sec:preliminaries}

This section outlines the governing equations of the quasi-static Biot consolidation model. 
We first introduce the classical two-field formulation involving displacement and pressure. 
Subsequently, to address the volumetric locking problem inherent in the incompressible limit, a three-field mixed formulation is introduced. 
Standard notations for Sobolev spaces and norms are adopted throughout this paper.

\subsection{Two-field Formulation of Biot's Model}

Consider the quasi-static Biot's model on a bounded domain $\Omega\subset\mathbb{R}^d$ ($d = 2$ or $3$) over the time interval $(0, T]$. 
The system couples the fluid flow and solid deformation and is governed by the following partial differential equations (PDEs):
\begin{subequations}\label{eq:2field}
    \begin{align}
        \partial_t\left(c_0p+\alpha\nabla\cdot\boldsymbol{u}\right)-\nabla\cdot (\boldsymbol{\kappa}\nabla p)=f, \quad &\text{in}~ \Omega\times(0,T],\label{eq:2field_p}\\
        -\nabla\cdot(\boldsymbol{\sigma}(\boldsymbol{u})-\alpha p\mathbb{I})=\boldsymbol{g}, \quad &\text{in}~ \Omega\times(0,T].\label{eq:2field_u}
    \end{align}
\end{subequations}
Eq. \eqref{eq:2field_p} represents the mass conservation of the fluid, while Eq. \eqref{eq:2field_u} represents the momentum balance of the solid matrix.

Here, the primary unknowns are the displacement vector of the solid matrix $\boldsymbol{u}$ and the fluid pressure $p$. 
The coefficients are defined as follows: $c_0> 0$ is the specific storage coefficient, $0<\alpha\le 1$ is the Biot-Willis coefficient, and $\boldsymbol{\kappa}$ is the hydraulic conductivity tensor. 
The total stress tensor $\boldsymbol{\sigma}_{total} = \boldsymbol{\sigma}(\boldsymbol{u}) - \alpha p \mathbb{I}$ governs the mechanics, where $\boldsymbol{\sigma}(\boldsymbol{u})$ is the effective stress tensor defined by the linear elasticity law:
\begin{equation}
    \boldsymbol{\sigma}(\boldsymbol{u})=2\mu\boldsymbol{\epsilon}(\boldsymbol{u})+\lambda(\nabla\cdot\boldsymbol{u})\mathbb{I},
\end{equation}
with $\boldsymbol{\epsilon}(\boldsymbol{u})=\frac{1}{2}[\nabla\boldsymbol{u}+(\nabla\boldsymbol{u})^T]$ being the strain tensor and $\mathbb{I}$ the identity tensor. 
The Lam\'{e} parameters $\lambda$ and $\mu$ relate to the Young's modulus $E$ and Poisson's ratio $\nu$ via:
\begin{equation}
    \lambda=\frac{E\nu}{(1+\nu)(1-2\nu)},\quad\mu=\frac{E}{2(1+\nu)}.
\end{equation}

To complete the model \eqref{eq:2field}, the initial conditions are prescribed as:
\begin{equation}
    \boldsymbol{u}(\boldsymbol{x},0)=\boldsymbol{u}_0(\boldsymbol{x}),\quad p(\boldsymbol{x},0)=p_0(\boldsymbol{x}),\quad\text{in}~\Omega\times\{0\}.
\end{equation}
The boundary $\partial\Omega$ is partitioned into Dirichlet and Neumann parts for both fields: $\partial\Omega=\Gamma_{\boldsymbol{u}D}\cup\Gamma_{\boldsymbol{u}N}=\Gamma_{pD}\cup\Gamma_{pN}$, with disjoint sets $\Gamma_{\boldsymbol{u}D}\cap\Gamma_{\boldsymbol{u}N}=\emptyset$ and $\Gamma_{pD}\cap\Gamma_{pN}=\emptyset$. The boundary conditions are given by:
\begin{subequations}
    \begin{align}
        \boldsymbol{u}&=\boldsymbol{u}_D\quad \text{on}~ \Gamma_{\boldsymbol{u}D}\times(0,T], \quad&(\boldsymbol{\sigma}(\boldsymbol{u})-\alpha p\mathbb{I})\boldsymbol{n}=\boldsymbol{\sigma}_N\quad \text{on}~ \Gamma_{\boldsymbol{u}N}\times(0,T],\\
        p&=p_D\quad \text{on}~ \Gamma_{pD}\times(0,T], \quad &-\boldsymbol{\kappa}\nabla p\cdot\boldsymbol{n}=q_N\quad \text{on}~\Gamma_{pN}\times(0,T],
    \end{align}
\end{subequations}
where $\boldsymbol{n}$ is the unit outward normal vector. Note that the sign in the flux boundary condition follows Darcy's law. 
For simplicity, we assume that the data $\boldsymbol{u}_D$, $\boldsymbol{\sigma}_N$, $p_D$, and $q_N$ are sufficiently smooth to ensure the well-posedness of the problem, and the theoretical analysis in this paper only considers homogeneous boundary conditions. 

\subsection{Three-field Formulation for Incompressible Limit}

Standard finite element methods are well known to suffer from volumetric locking and spurious pressure oscillations as the material approaches the incompressible limit, i.e., $\nu \to 0.5$ or equivalently $\lambda \to \infty$. 
Similar difficulties arise for basic Physics-Informed Neural Networks (PINNs), which often fail to deliver accurate solutions under such extreme parameter regimes. 
In particular, when the incompressibility constraint is not adequately enforced and $\nabla \cdot \boldsymbol{u}$ does not vanish, the term $\lambda \nabla \cdot \boldsymbol{u}$ becomes unbounded, resulting in numerical instability and solution blow-up in the momentum equation~\eqref{eq:2field_u}. 
To circumvent these issues, we adopt the three-field mixed formulation proposed in \cite{oyarzua2016locking}. 
Specifically, we introduce an additional variable, the total pressure $\xi$, defined by
\begin{equation}
\xi \coloneqq \alpha p - \lambda \nabla \cdot \boldsymbol{u}.
\end{equation}
By introducing $\xi$ as an independent unknown, the singular term $\lambda \nabla \cdot \boldsymbol{u}$ is eliminated from the momentum equation, 
thereby removing the explicit dependence on the unbounded Lamé parameter. The resulting three-field formulation can be written as
\begin{subequations}\label{eq:3field}
    \begin{align}
        \left(c_0+\frac{\alpha^2}{\lambda}\right)\partial_tp-\frac{\alpha}{\lambda}\partial_t\xi-\nabla\cdot (\boldsymbol{\kappa}\nabla p)&=f,\quad\text{in}~\Omega\times(0,T],\label{eq:3field_p}\\
        -\nabla\cdot(2\mu\boldsymbol{\epsilon}(\boldsymbol{u})-\xi\mathbb{I})&=\boldsymbol{g},\quad\text{in}~\Omega\times(0,T],\label{eq:3field_u}\\
        -\nabla\cdot\boldsymbol{u}+\frac{\alpha}{\lambda}p-\frac{1}{\lambda}\xi&=0,\quad\text{in}~\Omega\times(0,T].\label{eq:3field_xi}
    \end{align}
\end{subequations}
Within this formulation, Eq.~\eqref{eq:3field_xi} serves as a constraint that weakly enforces the constitutive relationship among the volumetric strain, pore pressure, and total pressure. As $\lambda \to \infty$, the coefficients $1/\lambda$ and $\alpha/\lambda$ vanish, and the system consistently 
recovers the incompressible limit $\nabla \cdot \boldsymbol{u} \to 0$ without inducing numerical instability or solution blow-up.
Except for modifying the Neumann boundary condition of $\boldsymbol{u}$ to $(2\mu\boldsymbol{\epsilon}(\boldsymbol{u})-\xi\mathbb{I})\boldsymbol{n}=\boldsymbol{\sigma}_N$, the remaining initial and boundary conditions of $\boldsymbol{u}$ and $p$ remain the same as those in the two-field formulation. 
The initial condition of the variable $\xi$ is given by $\xi_0=\alpha p_0-\lambda\nabla\cdot\boldsymbol{u}$. 
$\xi$ does not require independent boundary conditions as it is purely internal to the domain, defined implicitly by the kinematic fields.

Let $L^2(\Omega)$ denote the space of square-integrable functions on $\Omega$, and $H^k(\Omega)$ denote the Sobolev space of functions with derivatives up to order $k$ in $L^2(\Omega)$. 
The norms are denoted by $\|\cdot\|_{L^2}$ and $\|\cdot\|_{k}$ respectively, and the $L^2$-inner product is denoted by $(\cdot,\cdot)$. 
The solution pair for the two-field model \eqref{eq:2field} is typically sought in the space $(\boldsymbol{u}, p) \in [H^1(\Omega)]^d \times H^1(\Omega)$. 
For the three-field model \eqref{eq:3field}, due to the constraint equation, the solution triplet is sought in the space $(\boldsymbol{u}, p, \xi) \in [H^1(\Omega)]^d \times H^1(\Omega)\times L^2(\Omega)$.
For convenience, we specifically define the following function spaces for displacement, fluid pressure, and total pressure:
\begin{equation}
    \boldsymbol{V}\coloneqq\{\boldsymbol{v}\in[H^1(\Omega)]^d:\boldsymbol{v}=\boldsymbol{0} ~\text{on}~\Gamma_{\boldsymbol{u}D}\},~\boldsymbol{W}\coloneqq\{w\in H^1(\Omega):w=0~\text{on}~\Gamma_{pD}\}, ~\boldsymbol{Z}\coloneqq L^2(\Omega).
\end{equation}

\section{Methodology: The FS-ENGD-PINN Framework}\label{sec:methods}

This section details the methodology of the proposed FS-ENGD-PINN framework. 
To address the challenges of strong physical coupling, we first introduce the Fixed-Stress (FS) splitting strategy, which decomposes the original problem into contractive subproblems. 
To overcome the optimization difficulties in the parameter space, we then describe the Energy Natural Gradient Descent method. 
Integrating these two strategies, we present the unified network architecture and explicitly construct the loss functions for both two-field and three-field formulations. 
Finally, the section concludes with a detailed algorithm summary and a theoretical analysis of the method's convergence.

\subsection{Fixed-stress Splitting Strategy for Decoupling}\label{subsec:FS}
In this subsection, we introduce the fixed-stress iteration strategy in the PINN framework. 
First, determine the stabilization parameter $\beta_{FS}$ based on the dimension of the model and the Lam\'{e} coefficients $\lambda,\mu$. 
Given an initial guess for displacement and pressure $(p^0,\boldsymbol{u}^0)$, then alternately solve the following steps for two-field model:

\textbf{Step 1:} Based on $\boldsymbol{u}^n$ and $p^n$ obtained from the $n$-th iteration, solve the following subproblem to obtain $p^{n+1}$:
\begin{subequations}\label{eq:FS_2field_1}
    \begin{align}
        (c_0+\beta_{FS})\partial_tp-\nabla\cdot\boldsymbol{\kappa}\nabla p&=f-\alpha\nabla\cdot\partial_t\boldsymbol{u}^n+\beta_{FS}\partial_tp^n,~~\text{in}~~\Omega\times(0,T],\label{eq:FS_2field_p}\\
        p(\boldsymbol{x},0)&=p_0(\boldsymbol{x}),~~\text{in}~~\Omega\times\{0\},\label{eq:FS_2field_p_ic}\\
        p&=p_D~\text{on}~\Gamma_{pD}\times(0,T],\label{eq:FS_2field_p_dir}\\
        -\boldsymbol{\kappa}\nabla p\cdot\boldsymbol{n}&=q_{N},~~\text{on}~~\Gamma_{pN}\times(0,T].\label{eq:FS_2field_p_neu}
    \end{align}
\end{subequations}
\textbf{Step 2:} Given $p^{n+1}$ obtained from Step 1, solve the following subproblem to obtain $\boldsymbol{u}^{n+1}$:
\begin{subequations}\label{eq:FS_2field_2}
    \begin{align}
        -\nabla\cdot\boldsymbol{\sigma}(\boldsymbol{u})&=\boldsymbol{g}-\alpha\nabla p^{n+1}~~\text{in}~~\Omega\times(0,T],\label{eq:FS_2field_u}\\
        \boldsymbol{u}(\boldsymbol{x},0)&=\boldsymbol{u}_0(\boldsymbol{x}),~~\text{in}~~\Omega\times\{0\},\label{eq:FS_2field_u_ic}\\
        \boldsymbol{u}&=\boldsymbol{u}_D,~~\text{on}~~\Gamma_{\boldsymbol{u}D}\times(0,T],\label{eq:FS_2field_u_dir}\\
        \boldsymbol{\sigma n}-\alpha p\boldsymbol{n}&=\boldsymbol{\sigma}_N,~~\text{on}~~\Gamma_{\boldsymbol{u}N}\times(0,T].\label{eq:FS_2field_u_neu}
    \end{align}
\end{subequations}
Assume that $(p,\boldsymbol{u})\subset\boldsymbol{W}\times\boldsymbol{V}$ is the unique solution of equation \eqref{eq:2field}, and $(p^n,\boldsymbol{u}^n)\subset\boldsymbol{W}\times\boldsymbol{V}$ is the solution obtained at the $n$-th iteration of the FS scheme \eqref{eq:FS_2field_1}, \eqref{eq:FS_2field_2}. 
Define the errors at the $n$-th iteration as $e_p^n=p^n-p$ and $e_{\boldsymbol{u}}^n=\boldsymbol{u}^n-\boldsymbol{u}$. 
Then, for the fixed-stress splitting iteration for the two-field formulation, we have the following theorem. 

\begin{theorem}[Contraction of Fixed-Stress Splitting {\cite{mikelic2013convergence}}]\label{thm:FS_2field}
Let the stabilization parameter $\beta_{FS} \ge \frac{\alpha^2}{2(\lambda + 2\mu/d)}$. 
Then the fixed-stress iteration mapping
$
\mathcal{S} : (p^n,\boldsymbol{u}^n) \mapsto (p^{n+1},\boldsymbol{u}^{n+1})
$
defines a contraction on the space $\mathcal{H} = \boldsymbol{W} \times \boldsymbol{V}$. 
Moreover, as shown in \cite{cai2023combination}, the pressure error satisfies
\begin{equation}
    \int_{0}^T\!\!\int_{\Omega} (\partial_t e_p^{n+1})^2 \, d\boldsymbol{x}\, dt
    \le L_{FS}
    \int_{0}^T\!\!\int_{\Omega} (\partial_t e_p^{n})^2 \, d\boldsymbol{x}\, dt,
\end{equation}
with a contraction factor
$L_{FS} = \frac{\beta_{FS}}{\beta_{FS} + 2c_0} < 1 .
$ 
\end{theorem}

For the three-field formulation, given an initial iterate $(p^0,\boldsymbol{u}^0,\xi^0)$, 
the FS scheme alternates between solving the $p$-subproblem and the coupled $(\boldsymbol{u},\xi)$-subproblem. 
In contrast to the two-field case, the three-field system provides an intrinsic stabilization, 
and no additional parameter $\beta_{FS}$ is required.

\textbf{Step 1:} Based on $\xi^n$ from the $n$-th iteration, solve the following subproblem to obtain $p^{n+1}$:
\begin{subequations}\label{eq:FS_3field_1}
    \begin{align}
        \left(c_0+\frac{\alpha^2}{\lambda}\right)\partial_tp-\nabla\cdot(\boldsymbol{\kappa}\nabla p)&=f+\frac{\alpha}{\lambda}\partial_t\xi^n, ~~\text{in}~~\Omega\times(0,T],\label{eq:FS_3field_p}\\
        p(\boldsymbol{x},0)&=p_0(\boldsymbol{x})~~\text{in}~~\Omega\times\{0\},\\
        p&=p_D~\text{on}~\Gamma_{pD}\times(0,T],\\
        -\boldsymbol{\kappa}\nabla p\cdot\boldsymbol{n}&=q_{N}~~\text{on}~~\Gamma_{pN}\times(0,T].
    \end{align}
\end{subequations}

\textbf{Step 2:} Use $p^{n+1}$ obtained in Step 1 to solve the following subproblem to obtain $\boldsymbol{u}^{n+1}$ and $\xi^{n+1}$:
\begin{subequations}\label{eq:FS_3field_2}
    \begin{align}
        -\nabla\cdot(2\mu\boldsymbol{\epsilon}(\boldsymbol{u})-\xi\mathbb{I})&=\boldsymbol{g}, ~~\text{in}~~\Omega\times(0,T],\\
        -\nabla\cdot\boldsymbol{u}+\frac{\alpha}{\lambda}p^{n+1}-\frac{1}{\lambda}\xi&=0,~~\text{in}~~\Omega\times(0,T],\label{eq:FS_3field_u_const}\\
        \boldsymbol{u}(\boldsymbol{x},0)&=\boldsymbol{u}_0(\boldsymbol{x}),~~\text{in}~~\Omega\times\{0\},\\
        \xi(\boldsymbol{x},0)&=\alpha p_0(\boldsymbol{x})-\lambda\nabla\cdot\boldsymbol{u}_0(\boldsymbol{x}), ~~\text{in}~~\Omega\times\{0\},\\
        \boldsymbol{u}&=\boldsymbol{u}_D~~\text{on}~~\Gamma_{\boldsymbol{u}D}\times(0,T],\\
        (2\mu\boldsymbol{\epsilon}(\boldsymbol{u})-\xi\mathbb{I})\boldsymbol{n}&=\boldsymbol{\sigma}_N, ~~\text{on}~~\Gamma_{\boldsymbol{u}N}\times(0,T].        
    \end{align}
\end{subequations}

\begin{theorem}[Unconditional Stability of Three-field FS Scheme]\label{thm:FS_3field}
    The sequence $\{(p^n,\boldsymbol{u}^n,\xi^n)\}$ generated by the three-field FS iterative scheme \eqref{eq:FS_3field_1}-\eqref{eq:FS_3field_2} converges to the unique solution $(p,\boldsymbol{u},\xi)$ of \eqref{eq:3field} for any positive values of the Lam\'{e} parameters $\lambda$ and $\mu$. Specifically, there exists a constant $L=\frac{\alpha^2/\lambda}{\alpha^2/\lambda+2c_0}<1$ such that
    \begin{equation}
        \int_{0}^T\int_{\Omega}(\partial_te_p^{n+1})^2d\boldsymbol{x}dt\le L\int_0^T\int_\Omega(\partial_te_p^n)^2d\boldsymbol{x}dt.
    \end{equation}
\end{theorem}
\begin{proof}
    The FS scheme for the three-field formulation can be viewed as a special case of the two-field FS scheme with the stabilization parameter $\beta_{FS}=\frac{\alpha^2}{\lambda}$. 
    More specifically, the $p$-subproblem \eqref{eq:FS_3field_1} is equivalent to \eqref{eq:FS_2field_1} with this choice of $\beta_{FS}$, and the $(\boldsymbol{u},\xi)$-subproblem \eqref{eq:FS_3field_2} corresponds to \eqref{eq:FS_2field_2} when substituting the constraint \eqref{eq:FS_3field_u_const}.
    According to Theorem \ref{thm:FS_2field}, for any $\beta_{FS}\ge\frac{\alpha^2}{2(\frac{2\mu}{d}+\lambda)}$, the FS scheme converges. Since $\frac{\alpha^2}{\lambda}\ge\frac{\alpha^2}{2(\frac{2\mu}{d}+\lambda)}$ holds for all positive $\lambda$ and $\mu$, the convergence of the three-field FS scheme is guaranteed unconditionally.
\end{proof}

\subsection{Energy Natural Gradient Descent Optimization under PINN Framework}\label{subsec:ENGD}
Consider partial differential equations (PDEs) defined as follows:
\begin{subequations}
    \begin{align}
        \mathcal{L}v&=f,\quad \text{in}~\Omega\times (0,T],\\
        \mathcal{B}v&=g,\quad \text{on}~\partial\Omega\times (0,T],\\
        v(\boldsymbol{x},0)&=v_0(\boldsymbol{x}),\quad \text{in}~\Omega\times\{0\},
    \end{align}
\end{subequations}
where $\mathcal{L}$ is a differential operator which may be nonlinear, $\mathcal{B}$ is the boundary operator, $f$ is the source term, $g$ is the boundary data, and $v_0$ is the initial condition. 
In the PINN framework, we approximate the solution $v(\boldsymbol{x},t)$ using a deep neural network $v(\boldsymbol{x},t;\boldsymbol{\theta})$ with parameters $\boldsymbol{\theta}$. 
Specifically, we employ a fully connected feed-forward neural network (also known as a multi-layer perceptron (MLP)) with depth  $L$. 
Let $\mathbf{h}^0 = (\boldsymbol{x}, t)^T \in \mathbb{R}^{d+1}$ denote the input vector that concatenates the spatial and temporal coordinates. 
The forward propagation process transforms the input through $L-1$ hidden layers and one output layer, which is recursively defined as follows:
\begin{subequations}\label{eq:NN_structure}
    \begin{align}
        \mathbf{h}^l &= \sigma\left(\mathbf{W}^l \mathbf{h}^{l-1} + \mathbf{b}^l\right), \quad \text{for} \quad l=1, \dots, L-1, \\
        v(\boldsymbol{x},t;\boldsymbol{\theta}) &= \mathbf{W}^L \mathbf{h}^{L-1} + \mathbf{b}^L,
    \end{align}
\end{subequations}
where $\mathbf{W}^l \in \mathbb{R}^{N_l \times N_{l-1}}$ and $\mathbf{b}^l \in \mathbb{R}^{N_l}$ represent the weight matrix and bias vector of the $l$-th layer, respectively, with $N_l$ denoting the number of neurons in that layer. 
The set of all trainable parameters is denoted as $\boldsymbol{\theta} = \{ \mathbf{W}^l, \mathbf{b}^l \}_{l=1}^L$. 
The function $\sigma(\cdot)$ is a smooth, component-wise non-linear activation function.
In the context of PINNs, the hyperbolic tangent function (tanh) is often selected as the activation function due to its $C^\infty$smoothness property, ensuring the existence of non-zero higher-order derivatives required for computing the residuals of partial differential equations.

The loss function is constructed to enforce the PDEs, boundary conditions, and initial conditions:
\begin{align}
    \mathcal{J}(\boldsymbol{\theta}) &= w_f\mathcal{J}^{\text{PDE}}(\boldsymbol{\theta}) + w_b\mathcal{J}^{\text{BC}}(\boldsymbol{\theta}) + w_0\mathcal{J}^{\text{IC}}(\boldsymbol{\theta}),\notag\\
        &= \frac{w_f}{2N_f}\sum_{i=1}^{N_f}|\mathcal{L}v(\boldsymbol{x}_f^i,t_f^i;\boldsymbol{\theta}) - f(\boldsymbol{x}_f^i,t_f^i)|^2 + \frac{w_b}{2N_b}\sum_{i=1}^{N_b}|\mathcal{B}v(\boldsymbol{x}_b^i,t_b^i;\boldsymbol{\theta}) - g(\boldsymbol{x}_b^i,t_b^i)|^2 \notag\\
        &\quad + \frac{w_0}{2N_0}\sum_{i=1}^{N_0}|v(\boldsymbol{x}_0^i,0;\boldsymbol{\theta}) - v_0(\boldsymbol{x}_0^i)|^2,
\end{align}
where $N_f$, $N_b$, and $N_0$ are the numbers of collocation points for the PDE, boundary conditions, and initial conditions, respectively. 
The weights $w_f$, $w_b$, and $w_0$ balance the contributions of each term in the loss function.

To introduce the Energy Natural Gradient Descent (ENGD), we first interpret the neural network training from a geometric perspective. 
Let $X$ be the Hilbert space containing the solution. The neural network defines a parametrization map $\Phi: \mathbb{R}^P \to X$, mapping parameters $\boldsymbol{\theta}$ to functions $v_{\boldsymbol{\theta}} \coloneqq v(\cdot, \cdot; \boldsymbol{\theta})$. 
The image of this map constitutes a sub-manifold within the function space $X$, denoted as the model manifold:
\begin{equation*}
    \mathcal{M}_{\boldsymbol{\Theta}} \coloneqq \{ v_{\boldsymbol{\theta}} \in X \mid \boldsymbol{\theta} \in \mathbb{R}^P \}.
\end{equation*}
Then the generalized tangent space of the model manifold is denoted as:
\begin{equation}
    T_{\boldsymbol{\theta}}\mathcal{M}_{\boldsymbol{\Theta}} = \text{span}\left\{\frac{\partial v_{\boldsymbol{\theta}}}{\partial\theta_1},\dots,\frac{\partial v_{\boldsymbol{\theta}}}{\partial \theta_P}\right\}.
\end{equation}

We aim to minimize the loss function, which represents an energy functional $E: X \to \mathbb{R}$, i.e. $\mathcal{J}(\boldsymbol{\theta}) = E(v_{\boldsymbol{\theta}})$. 
For linear PDEs, this energy is typically quadratic. The second-order variation (Hessian) of the energy functional, denoted as $D^2E(v)$, defines a symmetric bilinear form on the function space. 
We denote this bilinear form as $a(\cdot, \cdot)$:
\begin{equation}
    a(u, w) \coloneqq D^2E(v_{\boldsymbol{\theta}})(u, w).
\end{equation}
Based on this geometry, the Energy Natural Gradient Descent is defined as follows.
\begin{definition}[Energy Natural Gradient Descent \cite{muller2023achieving}]
    Let $\nabla\mathcal{J}(\boldsymbol{\theta})$ be the standard Euclidean gradient. The energy natural gradient update direction is defined as:
    \begin{equation}\label{eq:E-NG}
        \nabla^E\mathcal{J}(\boldsymbol{\theta}) \coloneqq G^{+}(\boldsymbol{\theta})\nabla\mathcal{J}(\boldsymbol{\theta}),
    \end{equation}
    where $G^{+}(\boldsymbol{\theta})$ is the Moore-Penrose pseudo-inverse of the Energy Gram matrix $G(\boldsymbol{\theta}) \in \mathbb{R}^{P \times P}$. The entries of the Gram matrix correspond to the energy inner product of the tangent vectors:
    \begin{equation}
        G_{ij}(\boldsymbol{\theta}) \coloneqq a\left( \frac{\partial v_{\boldsymbol{\theta}}}{\partial \theta_i}, \frac{\partial v_{\boldsymbol{\theta}}}{\partial \theta_j} \right).
    \end{equation}
\end{definition}
\begin{remark}
    In fact, from the form of the natural energy gradient in \eqref{eq:E-NG}, ENGD can be interpreted as a preconditioned gradient descent method, where the preconditioner is given by the pseudo-inverse of the Energy Gram matrix.
\end{remark}

The main motivation for using ENGD lies in its ability to approximate the Newton direction in the function space, thereby accelerating convergence. 
This property is formally stated in the following theorem.

\begin{proposition}[Energy Natural Gradient Descent in Function Space \cite{muller2023achieving}]
    Let $\Pi_{T_{\boldsymbol{\theta}}\mathcal{M}_{\boldsymbol{\Theta}}}^{a}$ denote the orthogonal projection onto the tangent space $T_{\boldsymbol{\theta}}\mathcal{M}_{\boldsymbol{\Theta}}$ with respect to the energy inner product $a(\cdot, \cdot)$.
    \begin{enumerate}
        \item If $D^2E$ is coercive everywhere, then the direction of the energy natural gradient update corresponds to the projection of the Newton direction in the function space:
        \begin{equation}
            DP_{\boldsymbol{\theta}} \cdot \nabla^E\mathcal{J}(\boldsymbol{\theta}) = \Pi_{T_{\boldsymbol{\theta}}\mathcal{M}_{\boldsymbol{\Theta}}}^{D^2E(v_{\boldsymbol{\theta}})} \left( (D^2E(v_{\boldsymbol{\theta}}))^{-1} DE(v_{\boldsymbol{\theta}}) \right),
        \end{equation}
        where $DP_{\boldsymbol{\theta}}$ is the differential of the parametrization map $\boldsymbol{\theta} \mapsto v_{\boldsymbol{\theta}}$.
        \item Furthermore, if the energy functional $E$ is quadratic (i.e., $D^2E = a$) and admits a unique minimizer $v^*$, then the update direction corresponds to the projection of the true error $v_{\boldsymbol{\theta}} - v^*$ onto the tangent space:
        \begin{equation}
            DP_{\boldsymbol{\theta}} \cdot \nabla^E\mathcal{J}(\boldsymbol{\theta}) = \Pi_{T_{\boldsymbol{\theta}}\mathcal{M}_{\boldsymbol{\Theta}}}^{a} (v_{\boldsymbol{\theta}} - v^*).
        \end{equation}
    \end{enumerate}
\end{proposition}

\subsection{Network Architecture and Loss Functions of FS-ENGD-PINN}\label{subsec:architecture_loss}
In this subsection, we present the unified network architecture and explicitly construct the loss functions for both the two-field and three-field formulations within the FS-ENGD-PINN framework. 
For clarity, we denote the neural network approximations of displacement, fluid pressure, and total pressure as $\hat{p}(\boldsymbol{x},t;\boldsymbol{\theta}_p)$, $\hat{\boldsymbol{u}}(\boldsymbol{x},t;\boldsymbol{\theta}_{\boldsymbol{u}})$ for two-field model, and $\hat{p}(\boldsymbol{x},t;\boldsymbol{\theta}_p)$, $(\hat{\boldsymbol{u}},\hat{\xi})(\boldsymbol{x},t;\boldsymbol{\theta}_{\boldsymbol{u},\xi})$ for three-field model, respectively. Here, $\boldsymbol{\theta}_p$, $\boldsymbol{\theta}_{\boldsymbol{u}}$ (or $\boldsymbol{\theta}_{\boldsymbol{u},\xi}$) represent the trainable parameters of the corresponding networks. 
In fact, these networks can share the same architecture as defined in \eqref{eq:NN_structure}, differing only in their parameters:
\begin{equation}
    \boldsymbol{\theta}_p = \{ \mathbf{W}^l_p, \mathbf{b}^l_p \}_{l=1}^L, \quad \boldsymbol{\theta}_{\boldsymbol{u}} = \{ \mathbf{W}^l_{\boldsymbol{u}}, \mathbf{b}^l_{\boldsymbol{u}} \}_{l=1}^L, \quad \boldsymbol{\theta}_{\boldsymbol{u},\xi} = \{ \mathbf{W}^l_{\boldsymbol{u},\xi}, \mathbf{b}^l_{\boldsymbol{u},\xi} \}_{l=1}^L.
\end{equation}
Follow \eqref{eq:NN_structure}, the outputs of these networks are given by:
\begin{align}
    \hat{p}(\boldsymbol{x},t;\boldsymbol{\theta}_p) &= \mathbf{W}^L_p \mathbf{h}^{L-1}_p + \mathbf{b}^L_p,\\
    \hat{\boldsymbol{u}}(\boldsymbol{x},t;\boldsymbol{\theta}_{\boldsymbol{u}}) &= \mathbf{W}^L_{\boldsymbol{u}} \mathbf{h}^{L-1}_{\boldsymbol{u}} + \mathbf{b}^L_{\boldsymbol{u}},\\
    (\hat{\boldsymbol{u}},\hat{\xi})(\boldsymbol{x},t;\boldsymbol{\theta}_{\boldsymbol{u},\xi}) &= \mathbf{W}^L_{\boldsymbol{u},\xi} \mathbf{h}^{L-1}_{\boldsymbol{u},\xi} + \mathbf{b}^L_{\boldsymbol{u},\xi},
\end{align}
where $\mathbf{h}^{L-1}_p$, $\mathbf{h}^{L-1}_{\boldsymbol{u}}$, and $\mathbf{h}^{L-1}_{\boldsymbol{u},\xi}$ are the outputs of the last hidden layer for the respective networks.
Since our Biot model involves mixed boundary conditions, we decompose the boundary loss term into Dirichlet and Neumann components for clarity.

\textbf{Two-field FS-ENGD-PINN loss:} At the $n$-th iteration, we first construct the loss function for the $p$-subproblem \eqref{eq:FS_2field_1} as follows:
\begin{align}
    \mathcal{J}_p(\boldsymbol{\theta}_p) &= w_f\mathcal{J}_p^{\text{PDE}}(\boldsymbol{\theta}_{\boldsymbol{u}}^n,\boldsymbol{\theta}_p^n;\boldsymbol{\theta}_p) + w_{bD}\mathcal{J}_p^{\text{Dir}}(\boldsymbol{\theta}_p) + w_{bN}\mathcal{J}_p^{\text{Neu}}(\boldsymbol{\theta}_p) + w_0\mathcal{J}_p^{\text{IC}}(\boldsymbol{\theta}_p),\notag\\
        &= \frac{w_f}{2N_f}\sum_{i=1}^{N_f}\left|\left(c_0+\beta_{FS}\right)\partial_t\hat{p}(\boldsymbol{x}_f^i,t_f^i;\boldsymbol{\theta}_p) - \nabla\cdot\boldsymbol{\kappa}\nabla \hat{p}(\boldsymbol{x}_f^i,t_f^i;\boldsymbol{\theta}_p) - f(\boldsymbol{x}_f^i,t_f^i) \right.\notag\\
        &\quad\left. + \alpha\nabla\cdot\partial_t\hat{\boldsymbol{u}}^n(\boldsymbol{x}_f^i,t_f^i;\boldsymbol{\theta}_{\boldsymbol{u}}^n) - \beta_{FS}\partial_t\hat{p}^n(\boldsymbol{x}_f^i,t_f^i;\boldsymbol{\theta}_p^n)\right|^2 \notag\\
        &\quad + \frac{w_{bD}}{2N_{bD}}\sum_{i=1}^{N_{bD}}\left| \hat{p}(\boldsymbol{x}_{bD}^i,t_{bD}^i;\boldsymbol{\theta}_p) - p_D(\boldsymbol{x}_{bD}^i,t_{bD}^i) \right|^2 \notag\\
        &\quad + \frac{w_{bN}}{2N_{bN}}\sum_{i=1}^{N_{bN}}\left| -\boldsymbol{\kappa}\nabla \hat{p}(\boldsymbol{x}_{bN}^i,t_{bN}^i;\boldsymbol{\theta}_p)\cdot\boldsymbol{n}(\boldsymbol{x}_{bN}^i) - q_N(\boldsymbol{x}_{bN}^i,t_{bN}^i) \right|^2 \notag\\
        &\quad + \frac{w_0}{2N_0}\sum_{i=1}^{N_0}\left| \hat{p}(\boldsymbol{x}_0^i,0;\boldsymbol{\theta}_p) - p_0(\boldsymbol{x}_0^i) \right|^2,
\end{align}
where $\hat{\boldsymbol{u}}^n(\boldsymbol{x},t)$ and $\hat{p}^n(\boldsymbol{x},t)$ are the approximations obtained from the previous iteration. 
Then we use the ENGD optimizer to minimize $\mathcal{J}_p^n(\boldsymbol{\theta}_p)$ and obtain the updated pressure approximation $\hat{p}^{n+1}(\boldsymbol{x},t)$:
\begin{equation}
    \boldsymbol{\theta}_p^{n+1}= \boldsymbol{\theta}_p^n - \eta_p \nabla^E\mathcal{J}_p(\boldsymbol{\theta}_p^n)= \boldsymbol{\theta}_p^n - \eta_p G_p^{+}(\boldsymbol{\theta}_p^n)\nabla\mathcal{J}_p(\boldsymbol{\theta}_p^n),
\end{equation}
where $G_p(\boldsymbol{\theta}_p^n)_{ij}$ can be computed as:
\begin{align}
    (G_p(\boldsymbol{\theta}_p^n))_{ij}&=\frac{w_f}{N_f}\sum_{k=1}^{N_f}\left(\left(c_0+\beta_{FS}\right)\partial_t\left(\frac{\partial \hat{p}(\boldsymbol{x}_f^k,t_f^k;\boldsymbol{\theta}_p^n)}{\partial \theta_{p,i}}\right)+\boldsymbol{\kappa}\nabla\left(\frac{\partial \hat{p}(\boldsymbol{x}_f^k,t_f^k;\boldsymbol{\theta}_p^n)}{\partial \theta_{p,i}}\right)\right)\cdot\notag\\
    &\quad\left(\left(c_0+\beta_{FS}\right)\partial_t\left(\frac{\partial \hat{p}(\boldsymbol{x}_f^k,t_f^k;\boldsymbol{\theta}_p^n)}{\partial \theta_{p,j}}\right)+\boldsymbol{\kappa}\nabla\left(\frac{\partial \hat{p}(\boldsymbol{x}_f^k,t_f^k;\boldsymbol{\theta}_p^n)}{\partial \theta_{p,j}}\right)\right) \notag\\
    &\quad + \frac{w_{bD}}{N_{bD}}\sum_{k=1}^{N_{bD}}\left(\frac{\partial \hat{p}(\boldsymbol{x}_{bD}^k,t_{bD}^k;\boldsymbol{\theta}_p^n)}{\partial \theta_{p,i}}\right)\left(\frac{\partial \hat{p}(\boldsymbol{x}_{bD}^k,t_{bD}^k;\boldsymbol{\theta}_p^n)}{\partial \theta_{p,j}}\right) \notag\\
    &\quad + \frac{w_{bN}}{N_{bN}}\sum_{k=1}^{N_{bN}}\left(-\boldsymbol{\kappa}\nabla\frac{\partial \hat{p}(\boldsymbol{x}_{bN}^k,t_{bN}^k;\boldsymbol{\theta}_p^n)}{\partial \theta_{p,i}}\cdot\boldsymbol{n}\right)\left(-\boldsymbol{\kappa}\nabla\frac{\partial \hat{p}(\boldsymbol{x}_{bN}^k,t_{bN}^k;\boldsymbol{\theta}_p^n)}{\partial \theta_{p,j}}\cdot\boldsymbol{n}\right) \notag\\
    &\quad + \frac{w_0}{N_0}\sum_{k=1}^{N_0}\left(\frac{\partial \hat{p}(\boldsymbol{x}_0^k,0;\boldsymbol{\theta}_p^n)}{\partial \theta_{p,i}}\right)\left(\frac{\partial \hat{p}(\boldsymbol{x}_0^k,0;\boldsymbol{\theta}_p^n)}{\partial \theta_{p,j}}\right).
\end{align}

Next, we construct the loss function for the $\boldsymbol{u}$-subproblem \eqref{eq:FS_2field_2}:
\begin{align}
    \mathcal{J}_{\boldsymbol{u}}(\boldsymbol{\theta}_{\boldsymbol{u}}) &= w_f\mathcal{J}_{\boldsymbol{u}}^{\text{PDE}}(\boldsymbol{\theta}_{p}^{n+1};\boldsymbol{\theta}_{\boldsymbol{u}}) + w_{bD}\mathcal{J}_{\boldsymbol{u}}^{\text{Dir}}(\boldsymbol{\theta}_{\boldsymbol{u}}) + w_{bN}\mathcal{J}_{\boldsymbol{u}}^{\text{Neu}}(\boldsymbol{\theta}_{p}^{n+1};\boldsymbol{\theta}_{\boldsymbol{u}}) + w_0\mathcal{J}_{\boldsymbol{u}}^{\text{IC}}(\boldsymbol{\theta}_{\boldsymbol{u}}),\notag\\
        &= \frac{w_f}{2N_f}\sum_{i=1}^{N_f}\left| -\nabla\cdot\boldsymbol{\sigma}(\hat{\boldsymbol{u}}(\boldsymbol{x}_f^i,t_f^i;\boldsymbol{\theta}_{\boldsymbol{u}})) - \boldsymbol{g}(\boldsymbol{x}_f^i,t_f^i) + \alpha\nabla \hat{p}^{n+1}(\boldsymbol{x}_f^i,t_f^i;\boldsymbol{\theta}_{p}^{n+1}) \right|^2 \notag\\
        &\quad + \frac{w_{bD}}{2N_{bD}}\sum_{i=1}^{N_{bD}}\left| \hat{\boldsymbol{u}}(\boldsymbol{x}_{bD}^i,t_{bD}^i;\boldsymbol{\theta}_{\boldsymbol{u}}) - \boldsymbol{u}_D(\boldsymbol{x}_{bD}^i,t_{bD}^i) \right|^2 \notag\\
        &\quad + \frac{w_{bN}}{2N_{bN}}\sum_{i=1}^{N_{bN}}\left| (\boldsymbol{\sigma}(\hat{\boldsymbol{u}}(\boldsymbol{x}_{bN}^i,t_{bN}^i;\boldsymbol{\theta}_{\boldsymbol{u}})) - \alpha \hat{p}^{n+1}(\boldsymbol{x}_{bN}^i,t_{bN}^i;\boldsymbol{\theta}_{p}^{n+1})\mathbb{I})\cdot\boldsymbol{n}(\boldsymbol{x}_{bN}^i)\right.\notag\\
        &\quad\left. - \boldsymbol{\sigma}_N(\boldsymbol{x}_{bN}^i,t_{bN}^i) \right|^2 + \frac{w_0}{2N_0}\sum_{i=1}^{N_0}\left| \hat{\boldsymbol{u}}(\boldsymbol{x}_0^i,0;\boldsymbol{\theta}_{\boldsymbol{u}}) - \boldsymbol{u}_0(\boldsymbol{x}_0^i) \right|^2.
\end{align}
Similarly, we minimize $\mathcal{J}_{\boldsymbol{u}}(\boldsymbol{\theta}_{\boldsymbol{u}})$ using the ENGD optimizer to obtain the updated displacement approximation $\hat{\boldsymbol{u}}^{n+1}(\boldsymbol{x},t)$:
\begin{equation}
    \boldsymbol{\theta}_{\boldsymbol{u}}^{n+1} = \boldsymbol{\theta}_{\boldsymbol{u}}^n - \eta_{\boldsymbol{u}} \nabla^E\mathcal{J}_{\boldsymbol{u}}(\boldsymbol{\theta}_{\boldsymbol{u}}^n) = \boldsymbol{\theta}_{\boldsymbol{u}}^n - \eta_{\boldsymbol{u}} G_{\boldsymbol{u}}^{+}(\boldsymbol{\theta}_{\boldsymbol{u}}^n)\nabla\mathcal{J}_{\boldsymbol{u}}(\boldsymbol{\theta}_{\boldsymbol{u}}^n),
\end{equation}
where $G_{\boldsymbol{u}}(\boldsymbol{\theta}_{\boldsymbol{u}}^n)$ can be computed as:
\begin{align}
    (G_{\boldsymbol{u}}(\boldsymbol{\theta}_{\boldsymbol{u}}^n))_{ij}&=\frac{w_f}{N_f}\sum_{k=1}^{N_f}\left(\nabla\cdot\boldsymbol{\sigma}\left(\frac{\partial \hat{\boldsymbol{u}}(\boldsymbol{x}_f^k,t_f^k;\boldsymbol{\theta}_{\boldsymbol{u}}^n)}{\partial \theta_{\boldsymbol{u},i}}\right)\right)\cdot\left(\nabla\cdot\boldsymbol{\sigma}\left(\frac{\partial \hat{\boldsymbol{u}}(\boldsymbol{x}_f^k,t_f^k;\boldsymbol{\theta}_{\boldsymbol{u}}^n)}{\partial \theta_{\boldsymbol{u},j}}\right)\right) \notag\\
        &\quad + \frac{w_{bD}}{N_{bD}}\sum_{k=1}^{N_{bD}}\left(\frac{\partial \hat{\boldsymbol{u}}(\boldsymbol{x}_{bD}^k,t_{bD}^k;\boldsymbol{\theta}_{\boldsymbol{u}}^n)}{\partial \theta_{\boldsymbol{u},i}}\right)\cdot\left(\frac{\partial \hat{\boldsymbol{u}}(\boldsymbol{x}_{bD}^k,t_{bD}^k;\boldsymbol{\theta}_{\boldsymbol{u}}^n)}{\partial \theta_{\boldsymbol{u},j}}\right) \notag\\
        &\quad + \frac{w_{bN}}{N_{bN}}\sum_{k=1}^{N_{bN}}\left(\boldsymbol{\sigma}\left(\frac{\partial \hat{\boldsymbol{u}}(\boldsymbol{x}_{bN}^k,t_{bN}^k;\boldsymbol{\theta}_{\boldsymbol{u}}^n)}{\partial \theta_{\boldsymbol{u},i}}\right)\cdot\boldsymbol{n}\right)\cdot\left(\boldsymbol{\sigma}\left(\frac{\partial \hat{\boldsymbol{u}}(\boldsymbol{x}_{bN}^k,t_{bN}^k;\boldsymbol{\theta}_{\boldsymbol{u}}^n)}{\partial \theta_{\boldsymbol{u},j}}\right)\cdot\boldsymbol{n}\right) \notag\\
        &\quad + \frac{w_0}{N_0}\sum_{k=1}^{N_0}\left(\frac{\partial \hat{\boldsymbol{u}}(\boldsymbol{x}_0^k,0;\boldsymbol{\theta}_{\boldsymbol{u}}^n)}{\partial \theta_{\boldsymbol{u},i}}\right)\cdot\left(\frac{\partial \hat{\boldsymbol{u}}(\boldsymbol{x}_0^k,0;\boldsymbol{\theta}_{\boldsymbol{u}}^n)}{\partial \theta_{\boldsymbol{u},j}}\right).
\end{align}
and $\eta_{\boldsymbol{u}}$ is the learning rate for the displacement network determined through a line search step:
\begin{equation}
    \eta_{\boldsymbol{u}} = \arg\min_{0<\eta\le 1} \mathcal{J}_{\boldsymbol{u}}\left(\boldsymbol{\theta}_{\boldsymbol{u}}^n - \eta\nabla^E\mathcal{J}_{\boldsymbol{u}}(\boldsymbol{\theta}_{\boldsymbol{u}}^n)\right),
\end{equation}

\textbf{Three-field FS-ENGD-PINN loss:} Similarly, at the $n$-th iteration, we construct the loss function for the $p$-subproblem \eqref{eq:FS_3field_1}:
\begin{align}
    \mathcal{J}_p(\boldsymbol{\theta}_p) &= w_f\mathcal{J}_p^{\text{PDE}}(\boldsymbol{\theta}_{\boldsymbol{u},\xi}^n,\boldsymbol{\theta}_p^n;\boldsymbol{\theta}_p) + w_{bD}\mathcal{J}_p^{\text{Dir}}(\boldsymbol{\theta}_p) + w_{bN}\mathcal{J}_p^{\text{Neu}}(\boldsymbol{\theta}_p) + w_0\mathcal{J}_p^{\text{IC}}(\boldsymbol{\theta}_p),\notag\\
        &= \frac{w_f}{2N_f}\sum_{i=1}^{N_f}\left|\left(c_0+\frac{\alpha^2}{\lambda}\right)\partial_t\hat{p}(\boldsymbol{x}_f^i,t_f^i;\boldsymbol{\theta}_p) - \nabla\cdot\boldsymbol{\kappa}\nabla \hat{p}(\boldsymbol{x}_f^i,t_f^i;\boldsymbol{\theta}_p) \right.\notag\\
        &\quad\left. - f(\boldsymbol{x}_f^i,t_f^i) - \frac{\alpha}{\lambda}\partial_t\hat{\xi}(\boldsymbol{x}_f^i,t_f^i;\boldsymbol{\theta}_{\boldsymbol{u},\xi})\right|^2 \notag\\
        &\quad + \frac{w_{bD}}{2N_{bD}}\sum_{i=1}^{N_{bD}}\left| \hat{p}(\boldsymbol{x}_{bD}^i,t_{bD}^i;\boldsymbol{\theta}_p) - p_D(\boldsymbol{x}_{bD}^i,t_{bD}^i) \right|^2 \notag\\
        &\quad + \frac{w_{bN}}{2N_{bN}}\sum_{i=1}^{N_{bN}}\left| -\boldsymbol{\kappa}\nabla \hat{p}(\boldsymbol{x}_{bN}^i,t_{bN}^i;\boldsymbol{\theta}_p)\cdot\boldsymbol{n}(\boldsymbol{x}_{bN}^i) - q_N(\boldsymbol{x}_{bN}^i,t_{bN}^i) \right|^2 \notag\\
        &\quad + \frac{w_0}{2N_0}\sum_{i=1}^{N_0}\left| \hat{p}(\boldsymbol{x}_0^i,0;\boldsymbol{\theta}_p) - p_0(\boldsymbol{x}_0^i) \right|^2,
\end{align}
and the loss function for the $(\boldsymbol{u},\xi)$-subproblem \eqref{eq:FS_3field_2}:
\begin{align}
    \mathcal{J}_{\boldsymbol{u},\xi}(\boldsymbol{\theta}_{\boldsymbol{u},\xi}) &= w_f\left(\mathcal{J}_{\boldsymbol{u}}^{\text{PDE}}(\boldsymbol{\theta}_{p}^{n+1};\boldsymbol{\theta}_{\boldsymbol{u},\xi}) + \mathcal{J}_{\xi}^{\text{PDE}}(\boldsymbol{\theta}_{p}^{n+1};\boldsymbol{\theta}_{\boldsymbol{u},\xi})\right) + w_{bD}\mathcal{J}_{\boldsymbol{u}}^{\text{Dir}}(\boldsymbol{\theta}_{\boldsymbol{u},\xi}) \notag\\
        &\quad + w_{bN}\mathcal{J}_{\boldsymbol{u}}^{\text{Neu}}(\boldsymbol{\theta}_{p}^{n+1};\boldsymbol{\theta}_{\boldsymbol{u},\xi}) + w_0\left(\mathcal{J}_{\boldsymbol{u}}^{\text{IC}}(\boldsymbol{\theta}_{\boldsymbol{u},\xi}) +\mathcal{J}_{\xi}^{\text{IC}}(\boldsymbol{\theta}_{\boldsymbol{u},\xi}) \right),\notag\\
        &= \frac{w_f}{2N_f}\sum_{i=1}^{N_f}\left| -\nabla\cdot(2\mu\boldsymbol{\epsilon}(\hat{\boldsymbol{u}}(\boldsymbol{x}_f^i,t_f^i;\boldsymbol{\theta}_{\boldsymbol{u},\xi}))-\hat{\xi}(\boldsymbol{x}_f^i,t_f^i;\boldsymbol{\theta}_{\boldsymbol{u},\xi})\mathbb{I}) - \boldsymbol{g}(\boldsymbol{x}_f^i,t_f^i) \right|^2 \notag\\
        &\quad + \frac{w_f}{2N_f}\sum_{i=1}^{N_f}\left|-\nabla\cdot\hat{\boldsymbol{u}}(\boldsymbol{x}_f^i,t_f^i;\boldsymbol{\theta}_{\boldsymbol{u},\xi}) + \frac{\alpha}{\lambda}\hat{p}^{n+1}(\boldsymbol{x}_f^i,t_f^i;\boldsymbol{\theta}_{p}^{n+1})-\frac{1}{\lambda}\hat{\xi}(\boldsymbol{x}_f^i,t_f^i;\boldsymbol{\theta}_{\boldsymbol{u},\xi}) \right|^2 \notag\\
        &\quad + \frac{w_{bD}}{2N_{bD}}\sum_{i=1}^{N_{bD}}\left| \hat{\boldsymbol{u}}(\boldsymbol{x}_{bD}^i,t_{bD}^i;\boldsymbol{\theta}_{\boldsymbol{u},\xi}) - \boldsymbol{u}_D(\boldsymbol{x}_{bD}^i,t_{bD}^i) \right|^2 \notag\\
        &\quad + \frac{w_{bN}}{2N_{bN}}\sum_{i=1}^{N_{bN}}\left| (2\mu\boldsymbol{\epsilon}(\hat{\boldsymbol{u}}(\boldsymbol{x}_{bN}^i,t_{bN}^i;\boldsymbol{\theta}_{\boldsymbol{u},\xi}))-\hat{\xi}(\boldsymbol{x}_{bN}^i,t_{bN}^i;\boldsymbol{\theta}_{\boldsymbol{u},\xi})\mathbb{I})\cdot\boldsymbol{n}(\boldsymbol{x}_{bN}^i)\right.\notag\\
        &\quad\left. - \boldsymbol{\sigma}_N\left(\boldsymbol{x}_{bN}^i,t_{bN}^i\right) \right|^2 + \frac{w_0}{2N_0}\sum_{i=1}^{N_0}\left| \hat{\boldsymbol{u}}(\boldsymbol{x}_0^i,0;\boldsymbol{\theta}_{\boldsymbol{u},\xi}) - \boldsymbol{u}_0(\boldsymbol{x}_0^i) \right|^2\notag\\
        &\quad+ \frac{w_0}{2N_0}\sum_{i=1}^{N_0}\left| \hat{\xi}(\boldsymbol{x}_0^i,0;\boldsymbol{\theta}_{\boldsymbol{u},\xi}) - \xi_0(\boldsymbol{x}_0^i) \right|^2.
\end{align}
The Gram matrices $G_p(\boldsymbol{\theta}_p^n)$ and $G_{\boldsymbol{u},\xi}(\boldsymbol{\theta}_{\boldsymbol{u},\xi}^n)$ can be computed like the two-field model, and the ENGD updates for $\boldsymbol{\theta}_p$ and $\boldsymbol{\theta}_{\boldsymbol{u},\xi}$ follow the same procedure as previously described.

Then, the overall unified architecture of the FS-ENGD-PINN framework for both models is illustrated in Figure \ref{fig:architecture}. In this conceptual frame, the shared components are depicted with solid lines, representing the baseline two-field model that outputs $\hat{\boldsymbol{u}}$. Extensions specific to the three-field model, i.e., the auxiliary variable $\hat{\xi}$ and its computational pathways, are highlighted with dashed lines and bracketed notation $[\hat{\xi}]$. Note that $\hat{\boldsymbol{u}}[\hat{\xi}]$ refers to $\hat{\boldsymbol{u}}$ in the two-field model, and represents both $\hat{\boldsymbol{u}}$ and $\hat{\xi}$ in the three-field model.
\begin{figure}[!t]
\centering
\includegraphics[width=\linewidth]{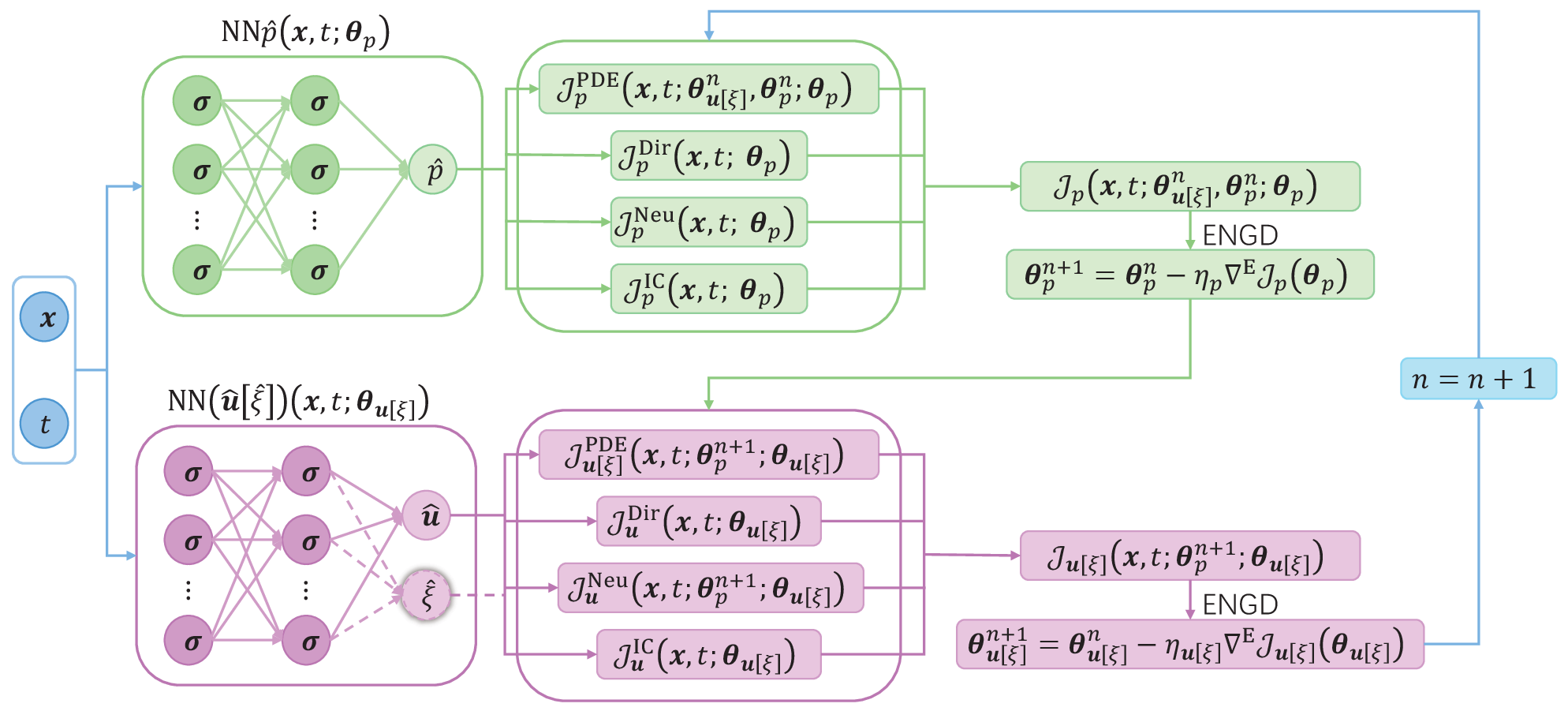}
\caption{The unified network architecture for the FS-ENGD-PINN framework.}\label{fig:architecture}
\end{figure}

\subsection{Convergence Analysis of the FS-ENGD-PINN Framework}\label{subsec:analysis}

In this subsection, we analyze the convergence properties of the proposed FS-ENGD-PINN framework. 
The analysis bridges the contraction mapping property of the fixed-stress splitting iterative scheme and the 
optimization efficiency of the Energy Natural Gradient Descent.
Since the three-field FS scheme is a special case of the standard FS splitting 
(with $\beta_{FS}=\frac{\alpha^2}{\lambda}$), we first present the analysis for the two-field formulation.

Let $(p^*, \boldsymbol{u}^*, \xi^*)$ denote the exact solution of the coupled Biot's system. 
At the $n$-th iteration, let $(\hat{p}^n, \hat{\boldsymbol{u}}^n, \hat{\xi}^n)$ denote the neural network approximations. 
To rigorously quantify the error, we introduce the reference solutions $(p^{n+1}, \boldsymbol{u}^{n+1}, \xi^{n+1})$ for the current iteration step. 
To align with the contraction property of the fixed-stress splitting scheme (Theorem \ref{thm:FS_2field}), we define the error metric for both fields based on the norms of their time derivatives. 
Specifically, we employ the $L^2$-norm of the time derivative for pressure, and the $H^1$-norm of the time derivative (velocity) for displacement:
\begin{equation}
    \| p \|_{\partial_t} \coloneqq \| \partial_t p \|_{L^2(0,T; L^2(\Omega))}, \quad \| \boldsymbol{u} \|_{\boldsymbol{V}_{\partial_t}} \coloneqq \| \partial_t \boldsymbol{u} \|_{L^2(0,T; H^1(\Omega)^d)}.
\end{equation}
Then, define the total error at iteration $n$ as 
\begin{equation}\label{eq:total_error}
    \mathcal{E}^n \coloneqq \| \hat{p}^n - p^* \|_{\partial_t} + \| \hat{\boldsymbol{u}}^n - \boldsymbol{u}^* \|_{\boldsymbol{V}_{\partial_t}}.
\end{equation}
Note that convergence in these norms implies convergence of the functions $p$ and $\boldsymbol{u}$ themselves, provided the initial conditions are satisfied.

In our framework, the ENGD optimizer is employed to minimize the loss functions $\mathcal{J}_p$ and $\mathcal{J}_{\boldsymbol{u}}$. 
The global minimum of the loss corresponds precisely to the reference solution $(p^{n+1}, \boldsymbol{u}^{n+1})$.
Relying on the Universal Approximation Theorem \cite{cybenko1989approximation, hornik1989multilayer}, we assume that the neural network possesses sufficient capacity to approximate the solution space. 
Consequently, the total error is dominated by the optimization residual, which allows us to establish the following precision guarantee \cite{de2024error}:

\begin{lemma}[Optimization Precision and Approximation Capability]\label{lemma:engd}
    Let $\mathcal{N}$ be a neural network class with sufficient width and depth. 
    According to the Universal Approximation Theorem, for any arbitrary tolerance $\epsilon > 0$, there exists an optimal network configuration within $\mathcal{N}$ such that the approximation error relative to the reference solution is bounded by $\epsilon$.
    With the ENGD optimizer effectively navigating the loss landscape, we define $\delta_{max}$ as the effective error bound in the strong norm $H^1([0, T]; H^1(\Omega))$ that accounts for both this theoretical approximation limit $\epsilon$ and the numerical optimization residual.
    Crucially, $\delta_{max} \to 0$ as the network capacity increases ($\epsilon \to 0$) and the optimization converges. 
    Since the norms defined above involving time derivatives are bounded by this strong norm, it follows that:

    \begin{equation}\label{eq:opt_bound}
        \| \hat{p}^{n+1} - p^{n+1} \|_{\partial_t} \le \delta_{max}, \quad \| \hat{\boldsymbol{u}}^{n+1} - \boldsymbol{u}^{n+1} \|_{\boldsymbol{V}_{\partial_t}} \le \delta_{max}.
    \end{equation}
    For the three-field case, a similar bound holds for the total pressure using the same norm: $\| \hat{\xi}^{n+1} - \xi^{n+1} \|_{\partial_t} \le \delta_{max}$.
\end{lemma}

\begin{remark}[Superiority and Trade-offs of ENGD vs. First-Order Optimizers]
The magnitude of $\delta_{max}$ inherently depends on the optimizer. First-order methods like Adam \cite{cai2023combination} rely on Euclidean geometry and frequently stagnate at a large residual error ($\delta_{Adam}$) within Biot's ill-conditioned loss landscape. Conversely, ENGD utilizes the Energy Gram matrix to approximate the Newton direction in function space, effectively mitigating stiffness and achieving a significantly tighter bound ($\delta_{max} \ll \delta_{Adam}$).
While constructing and approximating the Gram matrix inherently increases the computational cost and memory footprint per epoch, this overhead is heavily offset. As demonstrated in our numerical experiments, ENGD requires drastically fewer epochs to achieve high-fidelity convergence, thereby maintaining highly competitive overall computational efficiency.
\end{remark}

Based on the optimization bound and the physical stability of the fixed-stress splitting, we present the main convergence result:

\begin{theorem}[Convergence of FS-ENGD-PINN]\label{thm:convergence}
    Under the assumptions of Theorem \ref{thm:FS_2field} and Lemma \ref{lemma:engd}, the sequences of network approximations converge asymptotically in the time-derivative norms. Specifically:
    \begin{equation}
        \limsup_{n\to\infty} \| \hat{p}^n - p^* \|_{\partial_t} \le \frac{\delta_{max}}{1 - L_{FS}},
    \end{equation}
    \begin{equation}
        \limsup_{n\to\infty} \| \hat{\boldsymbol{u}}^n - \boldsymbol{u}^* \|_{\boldsymbol{V}_{\partial_t}} \le \left( 1 + \frac{C_u}{1 - L_{FS}} \right) \delta_{max}.
    \end{equation}
    Furthermore, for the three-field formulation, the total pressure $\xi$ also converges unconditionally:
    \begin{equation}
        \limsup_{n\to\infty} \| \hat{\xi}^n - \xi^* \|_{\partial_t} \le \delta_{max} + \frac{C_{\xi} \delta_{max}}{1 - L_{FS}},
    \end{equation}
    where $L_{FS} < 1$ is the contraction factor, and $C_u, C_{\xi}$ are coupling constants. Consequently, the total error $\mathcal{E}^n$ is bounded as $n \to \infty$.
\end{theorem}

\begin{proof}
    We prove the convergence by analyzing the stability of the splitting scheme and its interaction with the optimization error in a decoupled manner.
    
    \textbf{Step 1: Stability and Convergence of Pressure.}
    Let $\mathcal{S}$ denote the fixed-stress splitting operator. By definition, $(p^{n+1}, \boldsymbol{u}^{n+1}) = \mathcal{S}(\hat{p}^n, \hat{\boldsymbol{u}}^n)$ and $(p^*, \boldsymbol{u}^*) = \mathcal{S}(p^*, \boldsymbol{u}^*)$.
    According to Theorem \ref{thm:FS_2field}, the operator $\mathcal{S}$ is a contraction with respect to the pressure time-derivative:
    \begin{equation}\label{eq:contraction_p}
        \| p^{n+1} - p^* \|_{\partial_t} \le L_{FS} \| \hat{p}^n - p^* \|_{\partial_t},
    \end{equation}
    where $L_{FS} < 1$. 
    Combining this with the optimization bound $\| \hat{p}^{n+1} - p^{n+1} \|_{\partial_t} \le \delta_{max}$ from Lemma \ref{lemma:engd}, we obtain the error recurrence:
    \begin{equation}
        \| \hat{p}^{n+1} - p^* \|_{\partial_t} \le \| \hat{p}^{n+1} - p^{n+1} \|_{\partial_t} + \| p^{n+1} - p^* \|_{\partial_t} \le \delta_{max} + L_{FS} \| \hat{p}^n - p^* \|_{\partial_t}.
    \end{equation}
    This linear recurrence relation implies asymptotic convergence. Taking the limit superior as $n \to \infty$:
    \begin{equation}\label{eq:p_limit}
        \limsup_{n\to\infty} \| \hat{p}^n - p^* \|_{\partial_t} \le \frac{\delta_{max}}{1 - L_{FS}}.
    \end{equation}
    
    \textbf{Step 2: Stability of Displacement.}
    To derive a compatible bound for the displacement field, we consider the error equation. Let $\boldsymbol{e}_{\boldsymbol{u}}^{n+1} = \boldsymbol{u}^{n+1} - \boldsymbol{u}^*$ and $e_p^{n+1} = p^{n+1} - p^*$. Subtracting the exact momentum equation from the subproblem equation implies $-\nabla\cdot\boldsymbol{\sigma}(\boldsymbol{e}_{\boldsymbol{u}}^{n+1}) = -\alpha \nabla e_p^{n+1}$.
    Differentiating with respect to time yields $-\nabla\cdot\boldsymbol{\sigma}(\partial_t \boldsymbol{e}_{\boldsymbol{u}}^{n+1}) = -\alpha \nabla (\partial_t e_p^{n+1})$.
    Multiplying by $\partial_t \boldsymbol{e}_{\boldsymbol{u}}^{n+1}$, integrating over $\Omega$, and applying integration by parts, we obtain the weak form:
    \begin{equation}
        2\mu\int_\Omega \boldsymbol{\epsilon}(\partial_t \boldsymbol{e}_{\boldsymbol{u}}^{n+1}) : \boldsymbol{\epsilon}(\partial_t \boldsymbol{e}_{\boldsymbol{u}}^{n+1}) d\boldsymbol{x} +\lambda\int_\Omega (\nabla \cdot \partial_t \boldsymbol{e}_{\boldsymbol{u}}^{n+1})^2 d\boldsymbol{x} =\alpha  \int_\Omega (\partial_t e_p^{n+1}) \nabla\cdot (\partial_t \boldsymbol{e}_{\boldsymbol{u}}^{n+1}) d\boldsymbol{x}.
    \end{equation}
    Using Korn's inequality (coercivity) on the left-hand side and the Cauchy-Schwarz inequality on the right-hand side yields:
    \begin{equation}
        C_{coer} \| \partial_t \boldsymbol{e}_{\boldsymbol{u}}^{n+1} \|_{\boldsymbol{V}}^2 \le \alpha \| \partial_t e_p^{n+1} \|_{L^2(\Omega)} \| \nabla \cdot \partial_t \boldsymbol{e}_{\boldsymbol{u}}^{n+1} \|_{L^2(\Omega)}.
    \end{equation}
    Since $\| \nabla \cdot \boldsymbol{v} \|_{L^2} \le \sqrt{d} \| \boldsymbol{v} \|_{\boldsymbol{V}}$, we divide by $\| \partial_t \boldsymbol{e}_{\boldsymbol{u}}^{n+1} \|_{\boldsymbol{V}}$ to obtain:
    \begin{equation}\label{eq:stability_u}
        \| \partial_t \boldsymbol{e}_{\boldsymbol{u}}^{n+1} \|_{\boldsymbol{V}} \le C_u \| \partial_t e_p^{n+1} \|_{L^2(\Omega)},
    \end{equation}
    where $C_u = \frac{\alpha \sqrt{d}}{C_{coer}}$. Squaring and integrating over time gives the space-time norm estimate:
    \begin{equation}\label{eq:contraction_u_dt}
        \| \boldsymbol{u}^{n+1} - \boldsymbol{u}^* \|_{\boldsymbol{V}_{\partial_t}} \le C_u \| p^{n+1} - p^* \|_{\partial_t}.
    \end{equation}
    
    \textbf{Step 3: Convergence of Displacement.}
    Using the triangle inequality and the stability result from Step 2:
    \begin{align}
        \| \hat{\boldsymbol{u}}^{n+1} - \boldsymbol{u}^* \|_{\boldsymbol{V}_{\partial_t}} &\le \| \hat{\boldsymbol{u}}^{n+1} - \boldsymbol{u}^{n+1} \|_{\boldsymbol{V}_{\partial_t}} + \| \boldsymbol{u}^{n+1} - \boldsymbol{u}^* \|_{\boldsymbol{V}_{\partial_t}} \notag \\
        &\le \delta_{max} + C_u \| p^{n+1} - p^* \|_{\partial_t}.
    \end{align}
    Substituting the contraction property of pressure $\| p^{n+1} - p^* \|_{\partial_t} \le L_{FS} \| \hat{p}^n - p^* \|_{\partial_t}$:
    \begin{equation}
        \| \hat{\boldsymbol{u}}^{n+1} - \boldsymbol{u}^* \|_{\boldsymbol{V}_{\partial_t}} \le \delta_{max} + C_u L_{FS} \| \hat{p}^n - p^* \|_{\partial_t}.
    \end{equation}
    Finally, taking the limit superior as $n \to \infty$ and utilizing Eq. \eqref{eq:p_limit}:
    \begin{align}
        \limsup_{n\to\infty} \| \hat{\boldsymbol{u}}^n - \boldsymbol{u}^* \|_{\boldsymbol{V}_{\partial_t}} &\le \delta_{max} + C_u L_{FS} \left( \frac{\delta_{max}}{1 - L_{FS}} \right) \notag \\
        &= \delta_{max} \left( 1 + \frac{C_u L_{FS}}{1 - L_{FS}} \right).
    \end{align}

    \textbf{Step 4: Convergence of Total Pressure $\xi$.}
    For the three-field formulation, the reference solutions satisfy the algebraic constraint exactly: $\xi^{n+1} = \alpha p^{n+1} - \lambda \nabla \cdot \boldsymbol{u}^{n+1}$ (Eq. \eqref{eq:FS_3field_u_const}). The exact solution satisfies the same relation: $\xi^* = \alpha p^* - \lambda \nabla \cdot \boldsymbol{u}^*$.
    Therefore, the error in the reference total pressure is linearly dependent on the errors of reference pressure and displacement:
    \begin{equation}
        \xi^{n+1} - \xi^* = \alpha (p^{n+1} - p^*) - \lambda \nabla \cdot (\boldsymbol{u}^{n+1} - \boldsymbol{u}^*).
    \end{equation}
    Taking the time-derivative norm and applying the triangle inequality:
    \begin{equation}
        \| \xi^{n+1} - \xi^* \|_{\partial_t} \le \alpha \| p^{n+1} - p^* \|_{\partial_t} + \lambda \| \nabla \cdot (\partial_t \boldsymbol{e}_{\boldsymbol{u}}^{n+1}) \|_{L^2}.
    \end{equation}
    Using the bound $\| \nabla \cdot \boldsymbol{v} \|_{L^2} \le \sqrt{d} \| \boldsymbol{v} \|_{\boldsymbol{V}}$ and the stability estimate for displacement from Eq. \eqref{eq:stability_u} ($\| \partial_t \boldsymbol{e}_{\boldsymbol{u}}^{n+1} \|_{\boldsymbol{V}} \le C_u \| p^{n+1} - p^* \|_{\partial_t}$), we derive the bound for the reference solution:
    \begin{equation}
        \| \xi^{n+1} - \xi^* \|_{\partial_t} \le (\alpha + \lambda \sqrt{d} C_u) \| p^{n+1} - p^* \|_{\partial_t} = C_{\xi} \| p^{n+1} - p^* \|_{\partial_t},
    \end{equation}
    where $C_{\xi} = \alpha + \lambda \sqrt{d} C_u$.
    
    Finally, we extend this bound to the network approximation $\hat{\xi}^{n+1}$ using the optimization precision (Lemma \ref{lemma:engd}) and the triangle inequality:
    \begin{align}
        \| \hat{\xi}^{n+1} - \xi^* \|_{\partial_t} &\le \| \hat{\xi}^{n+1} - \xi^{n+1} \|_{\partial_t} + \| \xi^{n+1} - \xi^* \|_{\partial_t} \notag \\
        &\le \delta_{max} + C_{\xi} \| p^{n+1} - p^* \|_{\partial_t}.
    \end{align}
    Substituting the contraction property of pressure $\| p^{n+1} - p^* \|_{\partial_t} \le L_{FS} \| \hat{p}^n - p^* \|_{\partial_t}$ and taking the limit superior:
    \begin{equation}
        \limsup_{n\to\infty} \| \hat{\xi}^n - \xi^* \|_{\partial_t} \le \delta_{max} + C_{\xi} L_{FS} \left( \frac{\delta_{max}}{1 - L_{FS}} \right).
    \end{equation}
    This completes the proof.
\end{proof}

\begin{remark}[Training Dynamics and Practical PINN Quality]
    While Theorem \ref{thm:convergence} establishes asymptotic convergence, practical training entails intentional early-stage oscillations driven by dynamic point resampling to mitigate spatial overfitting. 
    Unlike standard fully-coupled PINNs that suffer from gradient pathologies and require complex adaptive weighting, FS-ENGD-PINN achieves stable, predominantly monotonic convergence for homogeneous media using simple static weights. 
    For highly heterogeneous media exhibiting sharp interfaces (e.g., the layered Terzaghi's problem), late-stage micro-oscillations may persist near the numerical noise floor. This specific challenge is robustly resolved by our unsupervised physical validation strategy (Section 4.2.5). 
    Ultimately, synergizing the FS split with ENGD alleviates global gradient conflicts, yielding a robust solver with simplified hyperparameter tuning.
\end{remark}

\section{Numerical Benchmarks}\label{sec:experiments}
In this section, we present a series of numerical benchmarks to validate the accuracy, efficiency, and robustness of the proposed FS-ENGD-PINN framework for solving Biot's equations.

\subsection{Implementation Details}\label{subsec:details}
\begin{itemize}
    \item Dynamic Sampling Strategy: To mitigate overfitting and enhance generalization, we employ a dynamic Monte Carlo sampling strategy within the FS-ENGD-PINN framework.
    Unlike static grid methods that rely on fixed collocation points, we regenerate the entire set of collocation points for PDE residuals, boundary conditions, and initial conditions from a uniform distribution at the start of each FS iteration.
    This stochastic approach ensures global minimization of the energy functional across the spatiotemporal domain.
    Unless otherwise specified, the number of collocation points is set to $N_f=2000, N_{bD}=N_{bN}=400$, and $N_0=500$.
    The evaluation set is uniformly sampled and its size is fixed at $N_{eval}=50^3=125000$.
    \item Network Architecture and Initialization: We employ fully-connected feedforward neural networks (FNNs) with the hyperbolic tangent (tanh) activation functions. 
    Network parameters are initialized using the Xavier (Glorot) scheme \cite{glorot2010understanding} to maintain variance consistency of activations and gradients.
    By default, separate networks for pressure and displacement are configured with two hidden layers of 32 neurons each, and loss weights are set to $w_f = w_{bD} = w_{bN} = w_0 = 1$.
    \item Efficient Gradient Computation: 
    To enable scalability, the Gram matrix $G(\boldsymbol{\theta})$ is computed via efficient Jacobian contraction $G = \frac{1}{N} J^T J$, avoiding the explicit formation of the per-sample outer product tensor.
    Furthermore, we compute the natural energy gradient update direction $\nabla^E\mathcal{J}(\boldsymbol{\theta})$ by solving a least-squares problem instead of directly calculating the pseudo-inverse of the Gram matrix:
    \begin{equation}
        \nabla^E\mathcal{J}(\boldsymbol{\theta}) = \arg\min_{\varphi\in\mathbb{R}^P}\|G(\boldsymbol{\theta})\varphi - \nabla\mathcal{J}(\boldsymbol{\theta})\|_{L^2}^2,
    \end{equation}
    this formulation significantly enhances numerical stability, particularly where the Gram matrix may be ill-conditioned or singular.
    \item Optimization Schedule: 
    The learning rate $\eta_\cdot$ is dynamically determined via a logarithmic line search in the interval $(0,1]$ at each ENGD update step:
    \begin{equation}
        \eta_\cdot = \arg\min_{0<\eta\le 1} \mathcal{J}_\cdot\left(\boldsymbol{\theta}_p^n - \eta\nabla^E\mathcal{J}_\cdot(\boldsymbol{\theta}_p^n)\right).
    \end{equation}
    We set the maximum number of FS iterations to 20, with 20 epochs per iteration for the ENGD sub-solver, unless otherwise specified.
    \item Baselines and Metrics: 
    We benchmark our framework against: (1) Standard FS-PINN using the Adam optimizer \cite{cai2023combination}, and (2) Coupled PINN using ENGD without splitting.
    Performance is quantified using the Relative $L^2$ Error and $H^1$ Error metrics, defined as follows:
    \begin{align*}
        \mathcal{E}_{L^2}(\hat{v},v_{exact}) &= \left(\frac{V}{N_{eval}}\sum_{i=1}^{N_{eval}} \left| \hat{v}(\boldsymbol{x}_i,t_i) - v_{exact}(\boldsymbol{x}_i,t_i) \right|^2\right)^{1/2},\\
        \mathcal{E}_{RelL^2}(\hat{v},v_{exact}) &= \frac{\mathcal{E}_{L^2}(\hat{v},v_{exact})}{\left(\frac{V}{N_{eval}}\sum_{i=1}^{N_{eval}} \left| v_{exact}(\boldsymbol{x}_i,t_i) \right|^2\right)^{1/2}},\\
            \mathcal{E}_{H^1}(\hat{v},v_{exact}) &= \left((\mathcal{E}_{L^2}(\hat{v},v_{exact}))^2 + (\mathcal{E}_{L^2}(\nabla \hat{v},\nabla v_{exact}))^2\right)^{1/2},
    \end{align*}
    All experiments are implemented in Python using the JAX library \cite{dataset} and executed on an NVIDIA RTX A6000 GPU.
\end{itemize}

\subsection{Validation on Standard Benchmarks (Two-field Model)}
In this subsection, we validate the proposed FS-ENGD-PINN framework on several standard benchmark problems for Biot's equations using the two-field formulation. 
Unless otherwise specified, the material parameters are set as follows: $E=1.0, \nu=0.3, \alpha=1.0, c_0=1.0, \boldsymbol{\kappa}=K\mathbb{I}$, where $K=1.0$.
All 2D domains are defined as $\Omega = [0,1]\times[0,1]$, and the simulation runs until $T_{final}=1.0$.
The boundaries are denoted as $\Gamma_1=\{(1,y);0\le y\le 1\},\Gamma_2=\{(x,0);0\le x\le 1\},\Gamma_3=\{(0,y);0\le y\le 1\},\Gamma_4=\{(x,1);0\le x\le 1\}$.
By default, Dirichlet boundary conditions are imposed on $\Gamma_1$ and $\Gamma_3$, while Neumann boundary conditions are applied on $\Gamma_2$ and $\Gamma_4$.
All initial and Dirichlet boundary conditions are derived directly from the exact solutions.
Consequently, we will only specify the source terms and Neumann boundary conditions in the specific examples below.
Furthermore, Example 1 includes a comparative study of the FS-ENGD-PINN performance against the baseline methods described in Subsection \ref{subsec:details}. 

\subsubsection{Example 1: 2D Linear Biot's Problem with Mixed BCs}
Consider the following source terms:
\begin{align*}
    f(x,y,t)&=(c_0+2K)\sin(x+y)e^t+\alpha(x+y),\\
    \boldsymbol{g}(x,y,t)&=(-(\lambda+2\mu)t+\alpha\cos(x+y)e^t)(1,1)^T.
\end{align*}
The exact solutions and corresponding Neumann boundary conditions are given by:
\begin{align*}
    p(x,y,t)&=\sin(x+y)e^t, \quad\boldsymbol{u}(x,y,t)=\frac{t}{2}(x^2,y^2)^T,\\
    (\boldsymbol{\sigma}(\boldsymbol{u}) - \alpha p \mathbb{I} )\cdot \boldsymbol{n} &= \boldsymbol{h},\quad~~~\nabla p\cdot\boldsymbol{n}=\cos(x+y)e^t(n_1+n_2) \text{ on } \Gamma_2\cup\Gamma_4,
\end{align*}
where $\boldsymbol{h}=2\mu(xn_1,yn_2)^Tt+\lambda(x+y)(n_1,n_2)^Tt-\alpha\sin(x+y)(n_1,n_2)^Te^t$.

Figure \ref{fig:2Dlinear} presents the comparison between the predicted and exact solutions for the pressure and displacement fields at $T_{final}$ using the FS-ENGD-PINN framework. 
The results exhibit excellent agreement with the analytical solutions, demonstrating the high accuracy of the proposed method.
Quantitative evaluation metrics are summarized in Table \ref{tab:comparison_2Dlinear}.
\begin{figure}[!htb]
\centering
\includegraphics[width=\linewidth]{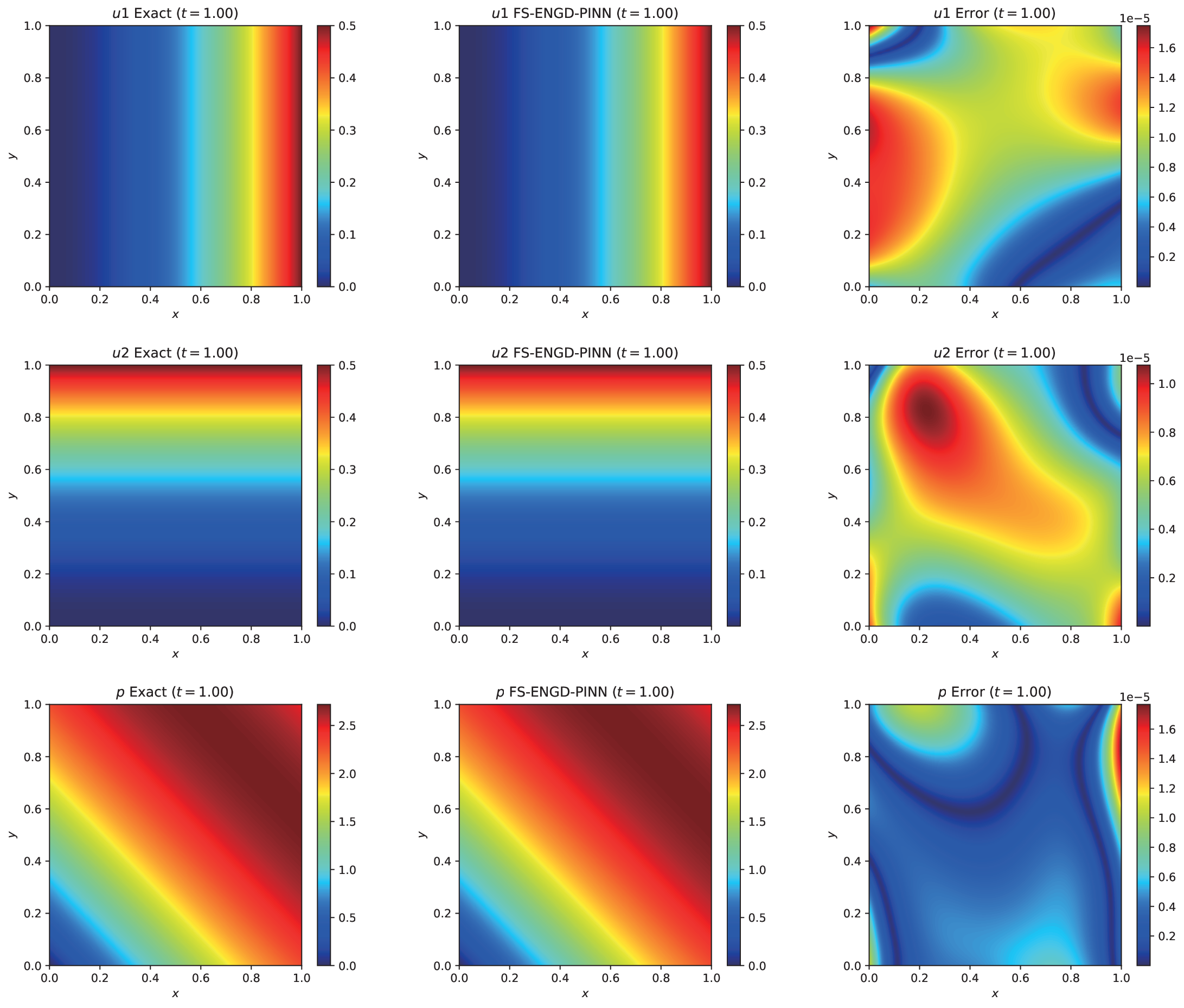}
\caption{Comparison of the predicted and exact solutions for Example 1.}\label{fig:2Dlinear}
\end{figure}

\textbf{Comparative Analysis:} 
To demonstrate the superiority of the proposed framework, we compare FS-ENGD-PINN against two baselines: the standard FS-PINN (optimized with Adam) and the coupled PINN with ENGD (ENGD-PINN).
For a fair comparison, the network architecture for all methods is set to 2 hidden layers with a width of 32 neurons.
The sampling and evaluation points remain consistent with those described in Subsection \ref{subsec:details}.
Following the standard FS-PINN setup in \cite{cai2023combination}, we configure the Adam optimizer with a learning rate of $10^{-4}$, set the number of FS iterations to 50, and train each subproblem for 400 epochs in each iteration. 
For the proposed FS-ENGD-PINN, as we perform 20 FS iterations with each network trained for 20 epochs per iteration, 
it results in a total computational budget equivalent to $20\times 20=400$ epochs.
Accordingly, the coupled ENGD-PINN is trained for 400 global epochs to ensure a fair comparison of computational cost ("Equivalent Epochs").

The quantitative results are summarized in Table \ref{tab:comparison_2Dlinear}, and the convergence histories are visualized in Figure \ref{fig:comparison_errors}.

\begin{table}[!htb]
\caption{Evaluation metrics and runtime for Example 1 using different methods.}
\label{tab:comparison_2Dlinear}
\centering
\resizebox{\linewidth}{!}{
    \begin{tabular}{lccccccc}
    \toprule
    \multirow{2}{*}{Methods} & \multicolumn{3}{c}{Rel$L^2$ Error} & \multicolumn{3}{c}{$H^1$ Error} & \multirow{2}{*}{Runtime (s)} \\
    \cmidrule(lr){2-4} \cmidrule(lr){5-7}
     & $p$ & $u_1$ & $u_2$ & $p$ & $u_1$ & $u_2$ & \\
    \midrule
    FS-ENGD-PINN      & \textbf{1.6950e-06} & \textbf{1.9896e-05} & \textbf{1.5344e-05} & \textbf{1.5530e-05} & \textbf{1.0000e-05} & \textbf{6.8510e-06} & 385.71 \\
    FS-PINN (Adam)  & 2.7701e-04 & 4.0661e-03 & 4.1858e-03 & 2.1867e-03 & 2.4430e-03 & 2.1091e-03 & \textbf{252.62} \\
    ENGD-PINN & 2.2705e-06 & 2.3769e-05 & 2.3140e-05 & 1.6010e-05 & 1.8710e-05 & 1.5072e-05 & 317.18 \\
    \bottomrule
\end{tabular}
}
\end{table}

\begin{figure}[!htb] 
    \centering
    \begin{minipage}[b]{0.48\textwidth} 
        \centering
        \includegraphics[width=\linewidth]{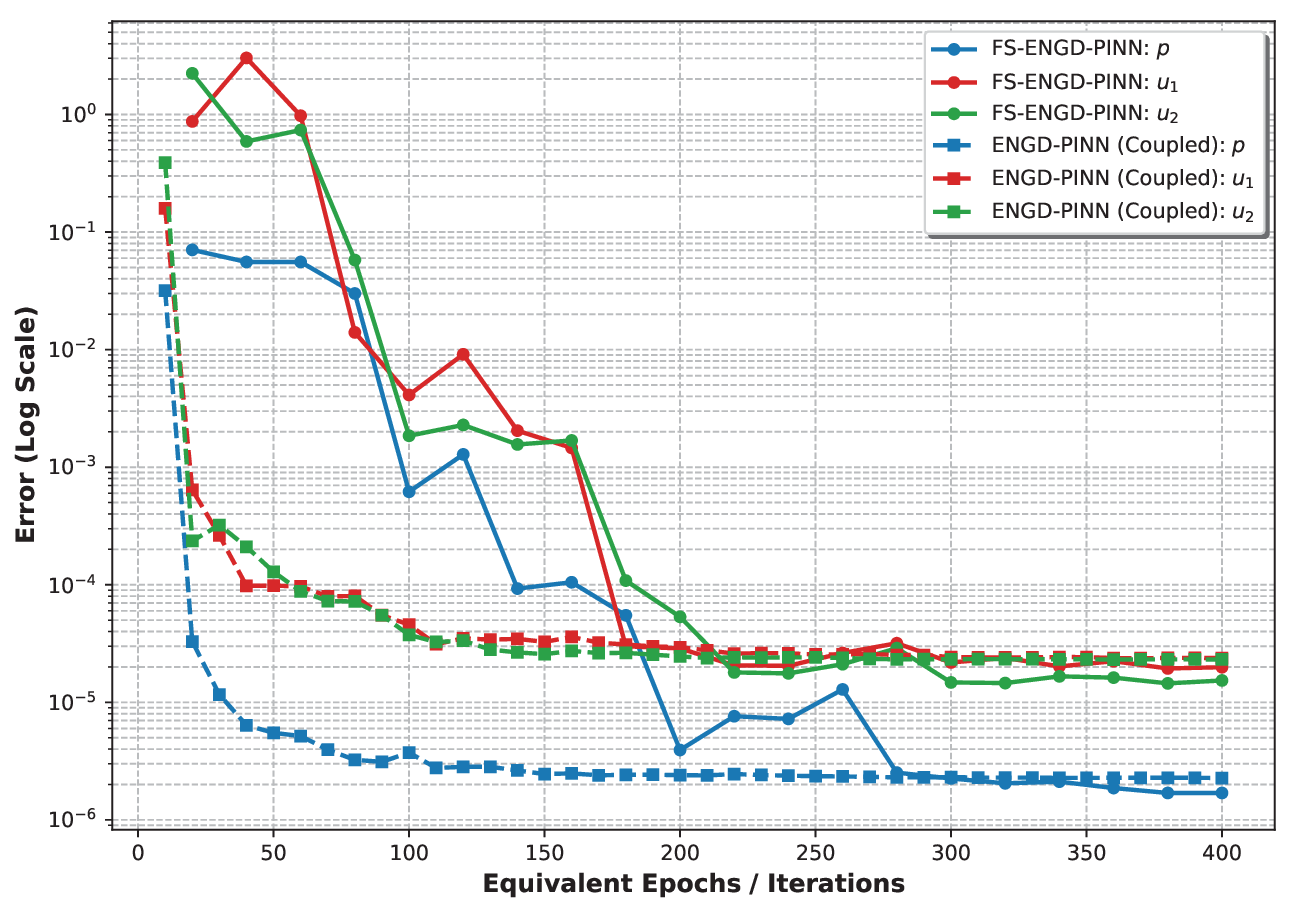}
        \centerline{(a) Convergence history of relative $L^2$ errors}
        \label{fig:comparison_RelL2_Error}
    \end{minipage}
    \hfill
    \begin{minipage}[b]{0.48\textwidth} 
        \centering
        \includegraphics[width=\linewidth]{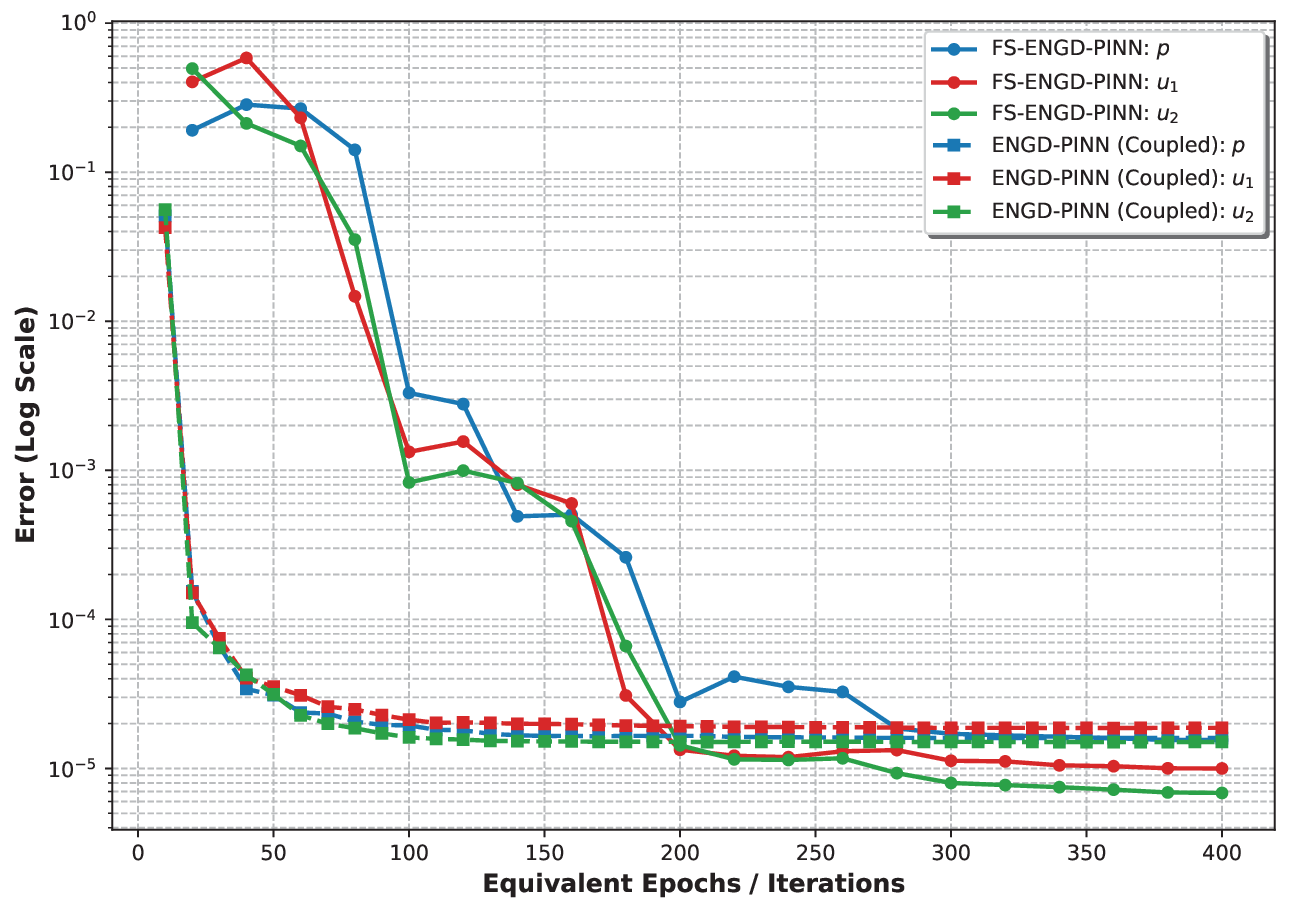}
        \centerline{(b) Convergence history of $H^1$ errors}
        \label{fig:comparison_H1_Error}
    \end{minipage}
    \caption{Convergence history of relative $L^2$ and $H^1$ errors over "Equivalent Epochs". (a) Convergence history of relative $L^2$ errors (left). (b) Convergence history of $H^1$ errors (right).}
    \label{fig:comparison_errors}
\end{figure}

\noindent\textbf{Discussion of Results.}
We first analyze the impact of the optimizer by comparing FS-PINN (Adam) and FS-ENGD-PINN. 
As shown in Table \ref{tab:comparison_2Dlinear}, although the standard FS-PINN with Adam records the shortest runtime (252.62 s), it stagnates at a relative $L^2$ error on the order of $10^{-4}$ for pressure and $10^{-3}$ for displacement, despite being trained for significantly more epochs. 
This highlights that the computational speed of first-order optimizers comes at the cost of accuracy; they struggle to handle the ill-conditioned loss landscape of Biot's equations, rendering the solution insufficient for high-precision requirements.

Next, we evaluate the efficiency-accuracy trade-off between the proposed FS-ENGD-PINN and the coupled ENGD-PINN.
While the total runtime of our method (385.71 s) is slightly higher than that of the coupled approach (317.18 s) for the full training cycle, this metric entails a significant gain in final accuracy.
More importantly, from the perspective of time-to-accuracy, our method demonstrates superior efficiency.
As visualized in Figure \ref{fig:comparison_errors}, the FS-ENGD-PINN surpasses the final converged accuracy of the coupled ENGD-PINN as early as the 15th FS iteration (equivalent to 300 epochs).
At this checkpoint, the cumulative runtime of our method is only \textbf{293.77 s}, which is faster than the total time required by the coupled ENGD-PINN.
This indicates that the physical decoupling provided by the Fixed-Stress scheme not only helps the network escape local minima but also allows it to reach a target high-fidelity threshold more efficiently than the monolithic approach.

\subsubsection{Example 2: 3D Linear Biot's Problem}
We extend the 2D configuration of Example 1 to a three-dimensional domain $\Omega = [0,1]^3$. The source terms are set to be:
\begin{align*}
    f(x,y,z,t)&=(c_0+3K)\sin(x+y+z)e^t+\alpha(x+y+z),\\
    \boldsymbol{g}(x,y,z,t)&=\left(-(\lambda+2\mu)t+\alpha\cos(x+y+z)e^t\right)(1,1,1)^T.
\end{align*}
The Dirichlet boundary conditions are applied on faces perpendicular to the $x$ and $y$ axes ($\Gamma_1 \cup \Gamma_2$), while Neumann conditions are applied on faces perpendicular to the $z$ axis ($\Gamma_3$). 
The exact solutions are given by:
\begin{align*}
    p(x,y,z,t)&=\sin(x+y+z)e^t, \quad\boldsymbol{u}(x,y,z,t)=\frac{t}{2}(x^2,y^2,z^2)^T.
\end{align*}
The corresponding Neumann boundary conditions on $\Gamma_3$ are derived as:
\begin{align*}
    (\boldsymbol{\sigma}(\boldsymbol{u}) - \alpha p \mathbb{I} )\cdot \boldsymbol{n} &= \boldsymbol{h},\quad \nabla p\cdot\boldsymbol{n}=\cos(x+y+z)e^t(n_1+n_2+n_3),
\end{align*}
where $\boldsymbol{h}=2\mu(xn_1,yn_2,zn_3)^T t+\lambda(x+y+z)(n_1,n_2,n_3)^T t-\alpha\sin(x+y+z)(n_1,n_2,n_3)^T e^t$.

For this 3D case, the number of collocation points is increased to $N_f=2000, N_{bD}=1000, N_{bN}=500, N_0=1000$, and the evaluation grid is $N_{eval}=160,000$.
Figure \ref{fig:3Dlinear} compares the predicted and exact solutions for the pressure and displacement fields at $T_{final}$. The FS-ENGD-PINN framework achieves excellent agreement with the analytical solution, demonstrating its capability to handle higher-dimensional problems effectively.

\begin{figure}[!htb]
\centering
\includegraphics[width=\linewidth]{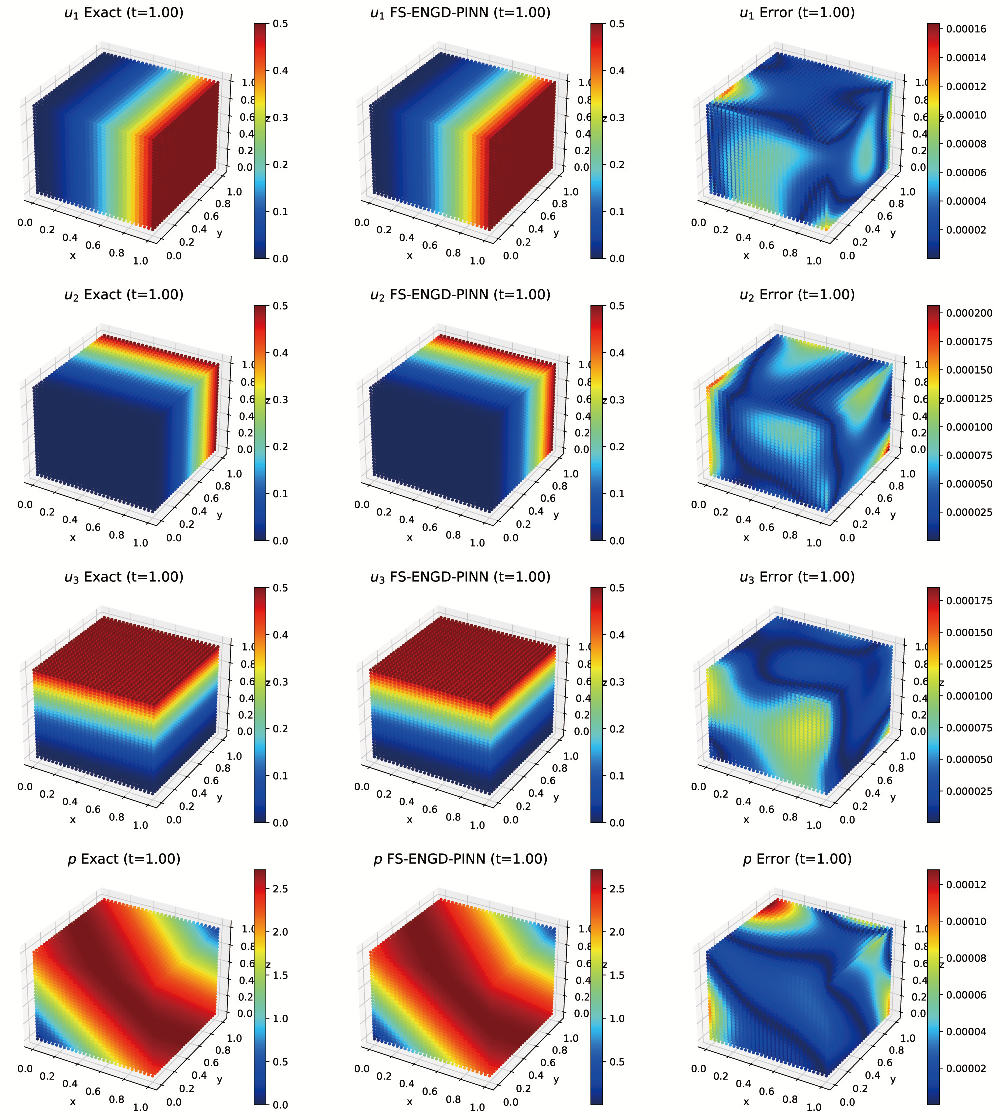}
\caption{Comparison of the predicted and exact solutions for Example 2.}\label{fig:3Dlinear}
\end{figure}

\subsubsection{Example 3: Problem with Complex Spatiotemporal Solution}
To test the solver's ability to capture intricate wave-like patterns, we construct a problem with trigonometric solutions. 
The source terms are defined as:
\begin{align*}
f(x,y,t) &= e^{-t} \left( (-c_0 + 2\pi^2 K) \sin(\pi x) \sin(\pi y) - \frac{\alpha \pi}{\mu + \lambda} \sin(\pi(x+y)) \right), \\[5pt]
\boldsymbol{g}(x,y,t) &= e^{-t} \begin{pmatrix} 
    \begin{aligned}
        &4\mu\pi^2 \sin(2\pi y)(2 \cos(2\pi x) - 1) \\
        &+ \left( \frac{2\mu\pi^2}{\mu+\lambda}\sin(\pi x) + \alpha\pi \cos(\pi x) \right) \sin(\pi y) - \pi^2 \cos(\pi(x+y))
    \end{aligned} 
    \\[20pt] 
    \begin{aligned}
        &4\mu\pi^2 \sin(2\pi x)(1 - 2 \cos(2\pi y)) \\
        &+ \left( \frac{2\mu\pi^2}{\mu+\lambda}\sin(\pi y) + \alpha\pi \cos(\pi y) \right) \sin(\pi x) - \pi^2 \cos(\pi(x+y))
    \end{aligned}
\end{pmatrix},
\end{align*}
and the exact solutions are:
\begin{align*}
    p(x,y,t)&=e^{-t}\sin(\pi x)\sin(\pi y),\\
    \boldsymbol{u}(x,y,t)&=e^{-t}\begin{pmatrix} 
        \sin(2\pi y)(\cos(2\pi x)-1)+\frac{1}{\lambda+\mu}\sin(\pi x)\sin(\pi y) \\ 
        \sin(2\pi x)(1 - \cos(2\pi y)) + \frac{1}{\lambda+\mu}\sin(\pi x)\sin(\pi y) \end{pmatrix}.
\end{align*}
The boundary conditions are configured to match the exact solution, with Neumann conditions for traction and flux applied on $\Gamma_2 \cup \Gamma_4$, and Dirichlet conditions elsewhere.
Figure \ref{fig:2Dcomplex} visualizes the results at $t=1.0$. 
Even with the complex spatiotemporal variations, the proposed method yields highly accurate predictions.


\begin{figure}[!htb]
\centering
\includegraphics[width=\linewidth]{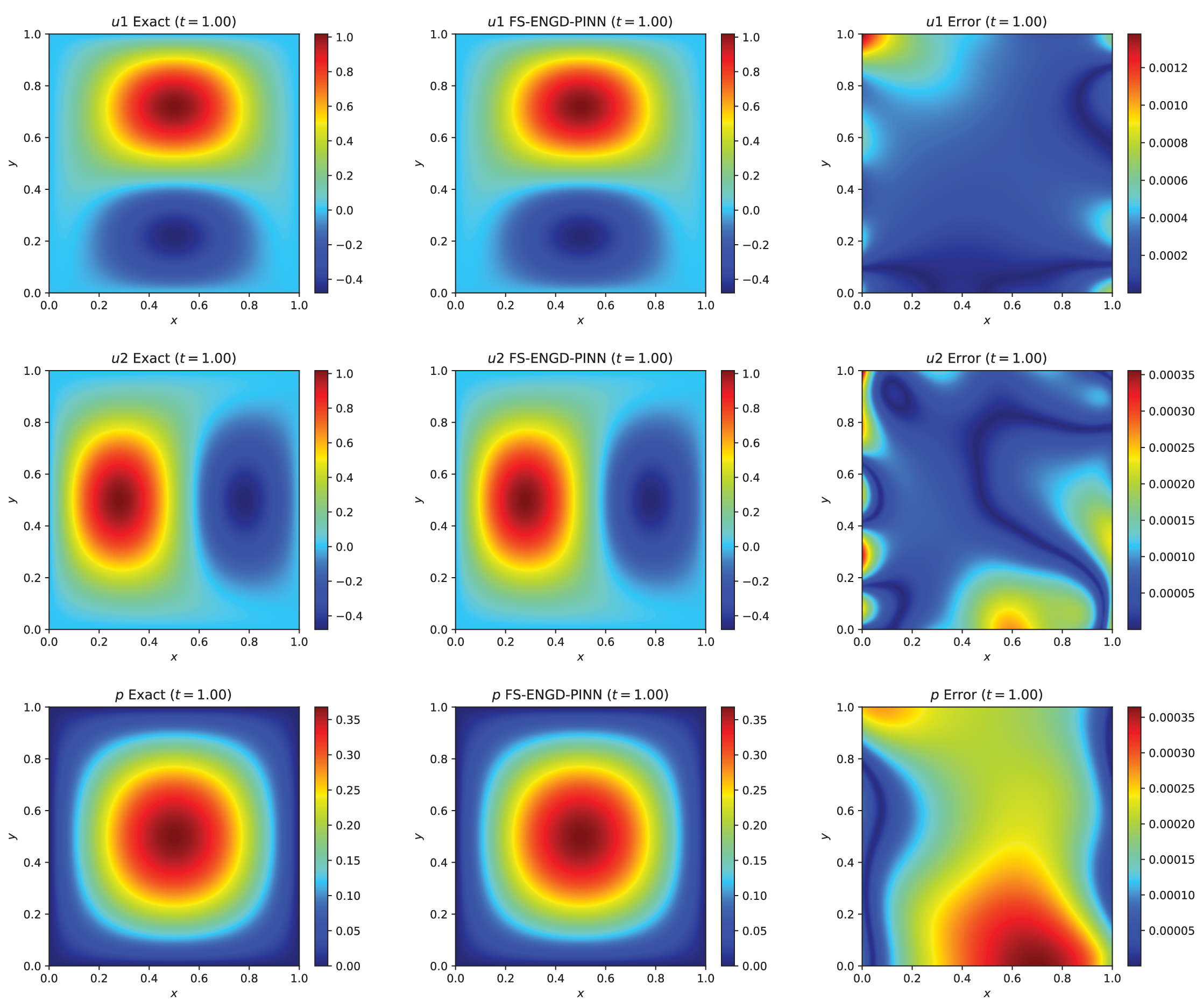}
\caption{Comparison of the predicted and exact solutions for Example 3.}\label{fig:2Dcomplex}
\end{figure}



Table \ref{tab:summary_ex23} summarizes the quantitative evaluation metrics for Examples 2-3. 
The proposed framework consistently achieves low relative $L^2$ and $H^1$ errors across 3D and complex scenarios, demonstrating its versatility and accuracy.

\begin{table}[!htb]
    \centering
    \caption{Summary of evaluation metrics and runtime for Example 2 and 3 using FS-ENGD-PINN.}
    \label{tab:summary_ex23}
    \setlength{\tabcolsep}{6pt}
    \begin{tabular}{lccccc}
        \toprule
        \multirow{2}{*}{Case} & \multicolumn{2}{c}{Rel$L^2$ Error} & \multicolumn{2}{c}{$H^1$ Error} & \multirow{2}{*}{Time (s)} \\
        \cmidrule(lr){2-3} \cmidrule(lr){4-5}
         & $p$ & $u_{mean}$ & $p$ & $u_{mean}$ & \\
        \midrule
        Ex 2 (3D Linear)   & 6.9702e-06 & 1.0996e-04 & 6.3861e-05 & 1.0237e-04 & 593.51 \\
        Ex 3 (Complex)     & 2.2094e-04 & 2.2388e-04 & 3.0916e-04 & 8.3799e-04 & 499.65 \\
        \bottomrule
    \end{tabular}
\end{table}

\subsubsection{Example 4: The Mandel Problem}
Consider the classical Mandel problem \cite{mandel1953consolidation, abousleiman1996mandel} to validate the solver's capability in capturing the non-monotonic \textit{Mandel-Cryer effect}. 
The physical setup consists of an infinitely long poroelastic slab sandwiched between two rigid, frictionless, and impermeable plates subjected to a constant vertical total load of $2F$.
Consequently, the effective load applied to the top boundary of the computational domain is $F$.
Due to symmetry, we simulate the first quadrant $\Omega = [0, 1] \times [0, 1]$. 
The exact solutions for displacement and pressure fields are given by:
\begin{align*}
    U_1(t)&=\left[\frac{F\nu}{2\mu}-\frac{F\nu_u}{\mu}\sum_{n=1}^{\infty}\frac{\sin\alpha_n\cos\alpha_n}{\alpha_n-\sin\alpha_n\cos\alpha_n}e^{-\alpha_n^2c_ft}\right]x\\
    &\quad+\frac{F}{\mu}\sum_{n=1}^{\infty}\frac{\cos\alpha_n\sin(\alpha_nx)}{\alpha_n-\sin\alpha_n\cos\alpha_n}e^{-\alpha_n^2c_ft},\\
    U_2(t)&=\left[\frac{-F(1-\nu)}{2\mu}+\frac{F(1-\nu_u)}{\mu}\sum_{n=1}^{\infty}\frac{\sin\alpha_n\cos\alpha_n}{\alpha_n-\sin\alpha_n\cos\alpha_n}e^{-\alpha_n^2c_ft}\right]y,\\
    P(t)&=\frac{2FB(1+\nu_u)}{3}\sum_{n=1}^{\infty}\frac{\sin\alpha_n}{\alpha_n-\sin\alpha_n\cos\alpha_n}(\cos(\alpha_nx)-\cos\alpha_n)e^{-\alpha_n^2c_ft},
\end{align*}
where $\alpha_n$ is the positive roots of $\tan\alpha_n=\frac{1-\nu_u}{\nu_u-\nu}\alpha_n$, and the other parameters are set as shown in Table \ref{tab:mandel_parameters}.

\begin{table}[!htb]
    \centering
    \caption{Physical parameters for the Mandel problem \cite{gu2024crank}.}
    \label{tab:mandel_parameters}
    \begin{tabular}{lcc}
        \toprule
        Symbol & Description & Value \\
        \midrule
        $T$ & Total simulation time & $1.0~s$ \\
        $\lambda$ & First Lamé parameter & $1.65 \times 10^9~Pa$ \\
        $\mu$ & Second Lamé parameter & $2.475 \times 10^9~Pa$ \\
        $\alpha$ & Biot-Willis constant & $1.0$ \\
        $c_0$ & Specific storage & $6.061 \times 10^{-11}~Pa^{-1}$ \\
        $K$ & Hydraulic conductivity & $9.869 \times 10^{-11}~m^2Pa^{-1}$ \\
        $F$ & Applied load & $6.0 \times 10^8~Nm^{-1}$ \\
        $B$ & Skempton's coefficient & $0.833$ \\
        $\nu_u$ & Undrained Poisson's ratio & $0.44$ \\
        $c_f$ & Fluid diffusivity & $0.465~m^2s^{-1}$ \\
        \bottomrule
    \end{tabular}
\end{table}

The initial and boundary conditions are set as follows:
\begin{align*}
    u_1(0)=\frac{F\nu_u}{2\mu}x,\quad u_2(0)&=\frac{-F(1-\nu_u)}{2\mu}y,\quad p(0)=\frac{FB(1+\nu_u)}{3},\quad\text{in}~\Omega\times\{0\},\\
    p&=0,\quad\boldsymbol{\sigma}_{11}=0, \quad\text{on}~\Gamma_1\times(0,T],\\
    \frac{\partial p}{\partial y}&=0,\quad\boldsymbol{\sigma}_{12}=0,\quad u_2=0,\quad\text{on}~\Gamma_2\times(0,T],\\
    \frac{\partial p}{\partial x}&=0,\quad\boldsymbol{\sigma}_{12}=0,\quad u_1=0,\quad\text{on}~\Gamma_3\times(0,T],\\
    \frac{\partial p}{\partial y}&=0,\quad\boldsymbol{\sigma}_{12}=0,\quad u_2=U_2,\quad\text{on}~\Gamma_4\times(0,T].
\end{align*}

\textbf{Normalization Strategy:} Given the vast disparity in physical parameters ($\lambda,\mu\propto 10^9$ vs. $c_0,K\propto10^{-11}$), direct training leads to severe gradient imbalances. 
To mitigate this numerical stiffness, we solve for the non-dimensionalized pressure $\hat{p}$ and displacements $\hat{\boldsymbol{u}}$ normalized by the applied load $F$:
\begin{align*}
    \hat{p} = p/F, \quad \hat{\boldsymbol{u}} = \boldsymbol{u} \cdot (\mu/F).
\end{align*}
This transformation scales the solution space to $O(1)$.
Furthermore, the PDE coefficients are scaled accordingly to balance the loss terms to $O(1)$, ensuring stable convergence.
We employ networks of sizes $[3,48,48,2]$ for $\hat{\boldsymbol{u}}$ and $[3,48,48,1]$ for $\hat{p}$, trained over 50 FS iterations with 30 epochs per subproblem.
The weights for the loss function are chosen as $w_f=1.0,w_{bD}=100.0,w_{bN}=100.0,w_0=5$.

Figure \ref{fig:mandel} shows the agreement between predicted and exact normalized solutions at $t=1.0$.
Crucially, Figure \ref{fig:pressure_analysis} tracks the pressure evolution, accurately capturing the initial rise and subsequent drop characteristic of the Mandel-Cryer effect.
\begin{figure}[!htb]
\centering
\includegraphics[width=\linewidth]{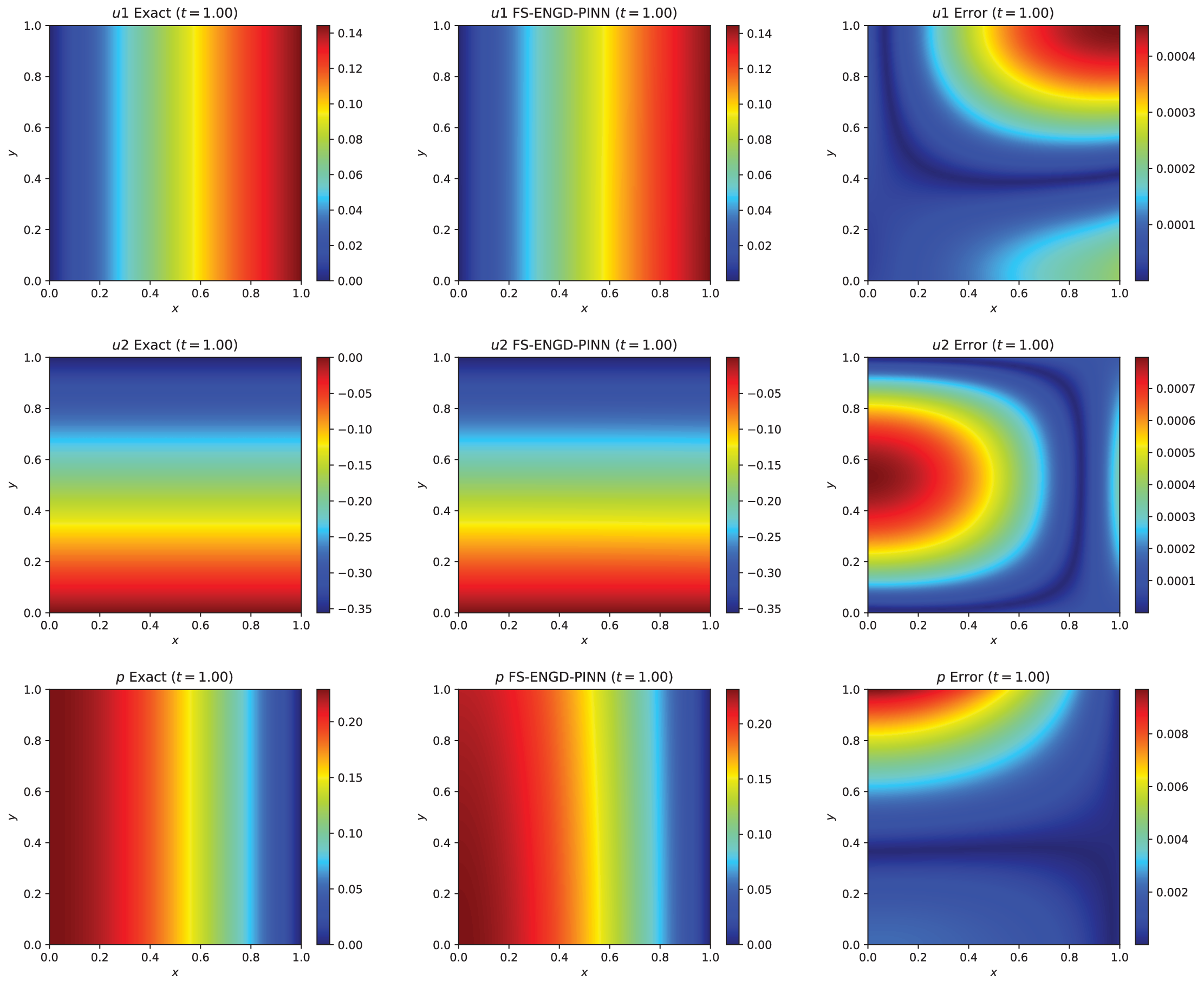}
\caption{Comparison of the normalized predicted and exact solutions for Example 4.}\label{fig:mandel}
\end{figure}

\begin{figure}[!htb]
\centering
\includegraphics[width=12cm]{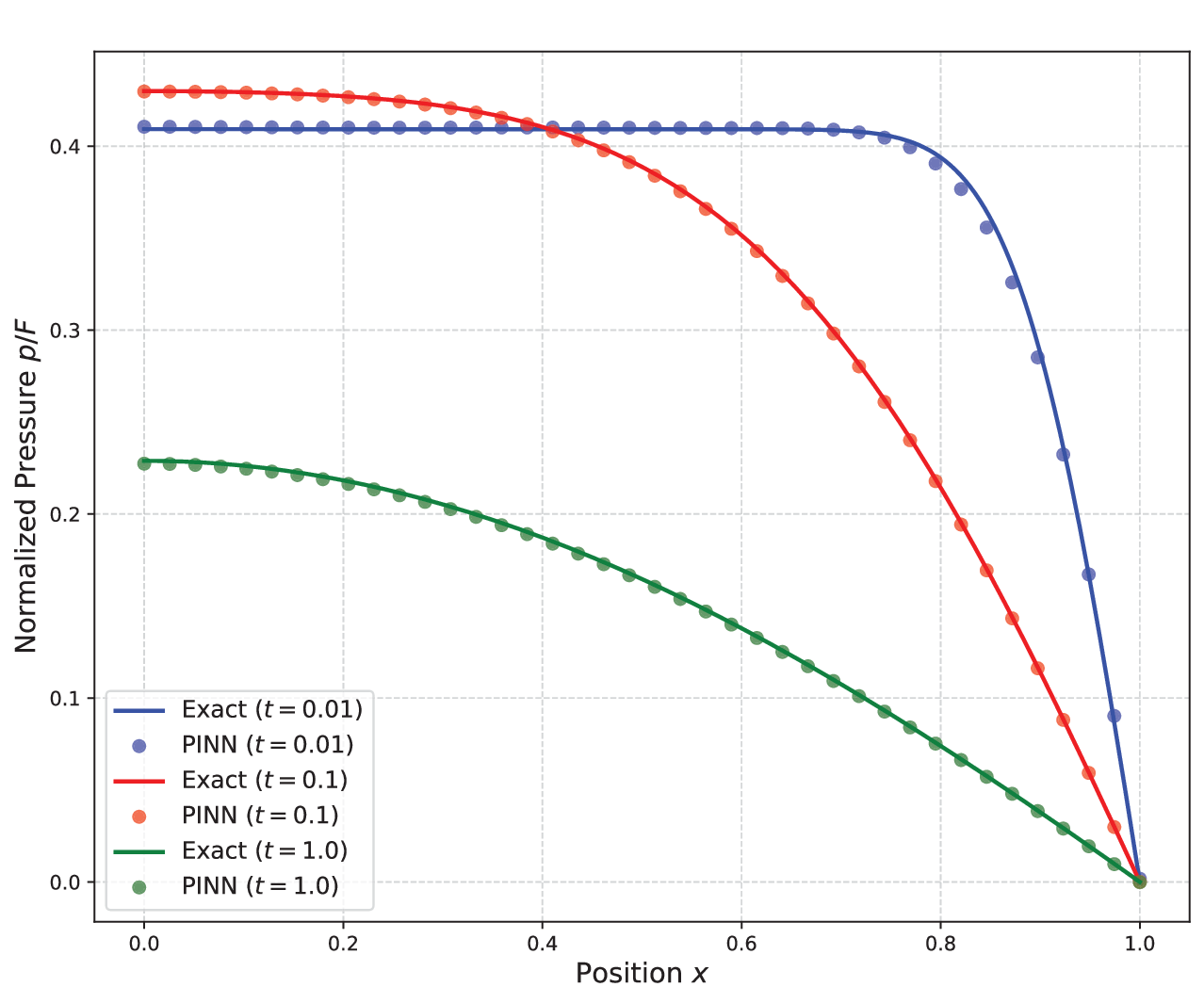}
\caption{Temporal evolution of normalized pressure for Example 4, capturing the Mandel-Cryer effect.}\label{fig:pressure_analysis}
\end{figure}

\subsubsection{Example 5: Layered Terzaghi's Problem}
To assess the performance of the proposed FS-ENGD-PINN framework in handling discontinuous material properties, we consider a two-layered Terzaghi consolidation problem \cite{lee1992study}.
Following the well-established benchmark practices in computational poromechanics for validating 1D consolidation within multi-dimensional fully coupled frameworks \cite{white2008stabilized, walker2023multiphysics}, we formulate the problem within a 2D rectangular domain $\Omega = (0,1) \times (0,1)$, which is partitioned into a top layer $\Omega_1=(0,1) \times (0.5,1)$ and a bottom layer $\Omega_2=(0,1) \times (0,0.5)$.
To recover the classical 1D uniaxial strain state and vertical fluid flow, symmetric roller boundaries ($u_1 = 0$) and zero-flux conditions ($\nabla p\cdot\boldsymbol{n}=0$) are rigorously enforced on the left and right boundaries $\Gamma_1 \cup \Gamma_3$. 
A constant compressive traction $\boldsymbol{\sigma}_{22}=-q_0$ is applied on the top boundary $\Gamma_4$ with a fully drained condition ($p=0$), while the bottom boundary $\Gamma_2$ is fixed ($\boldsymbol{u} = \mathbf{0}$) and impermeable ($\nabla p\cdot\boldsymbol{n}=0$).
We assume $\alpha=1.0, c_0=0.0$.
The material parameters are set to $\lambda_1=1.0, \mu_1=0.5, \kappa_1=1.0$ in $\Omega_1$, and $\lambda_2=2.0, \mu_2=1.0, \kappa_2=0.05$ in $\Omega_2$. This yields a permeability ratio of 20:1 across the material interface $y=0.5$.

\begin{remark}
    The material parameters chosen herein are of order $\mathcal{O}(1)$. Unlike the Mandel problem, which addresses global ill-conditioning from extreme scale disparities, this benchmark tests the solver's robustness against a local loss of regularity. The $20:1$ permeability jump strictly enforces a spatial weak discontinuity, where the pressure field $p$ remains $C^0$-continuous but exhibits a sharp gradient jump $\nabla p$ to maintain flux continuity.
\end{remark}

Directly fitting such sharp gradient jumps using standard MLPs with smooth activation functions (e.g., $\tanh$) inherently induces severe spectral bias near the interface. To circumvent this, the input layer is augmented with three analytical \textit{ans\"atze} to explicitly embed the singular structures into the network architecture:
\begin{enumerate}
    \item \textbf{Flux-preserving kink} $\phi_{kink}(y)$: Defined as $0.05(y-0.5)$ for $y > 0.5$ and $1.0(y-0.5)$ for $y \le 0.5$. This exactly encodes the $1:20$ pressure gradient discontinuity mandated by the permeability ratio, relieving the optimizer from discovering the $C^0$ kink from scratch.
    \item \textbf{Interface indicator} $\phi_{step}(y) = \text{sigmoid}((y-0.5)/0.02)$: A differentiable soft-step function localizing the material transition, which aids the network in decoupling the distinct mechanical responses in $\Omega_1$ and $\Omega_2$.
    \item \textbf{Early-time boundary layer} $\phi_{bl}(y,t) = \text{erf}\left((1-y)/\sqrt{8t+10^{-10}}\right)$: Derived from the classical similarity solution of the 1D diffusion equation on a semi-infinite domain \cite{wang2000theory}. 
    For 1D Terzaghi consolidation, the similarity variable typically incorporates the denominator $\sqrt{4c_{v}t}$. 
    In our formulation, the consolidation coefficient of the upper layer $\Omega_1$ is analytically computed as $c_{v,1} = \kappa_1 / (c_0 + \alpha^2/(\lambda_1+2\mu_1)) = 2.0$, which yields exactly $\sqrt{4c_{v,1}t} = \sqrt{8t}$. 
    A small numerical tolerance of $10^{-10}$ is introduced to prevent division by zero at $t=0$. This analytical feature strictly captures the infinite spatial derivative at the drained top boundary ($y=1$) as $t \to 0^+$, thereby effectively neutralizing the initial temporal singularity.
\end{enumerate}

Feeding these augmented features $\boldsymbol{h}_{in} = (x, y, t, \phi_{kink}, \phi_{step}, \phi_{bl})$ maps the ill-conditioned target solution onto a smoother, easily learnable residual manifold.

For the network configuration, both fields utilize 2 hidden layers with 40 neurons each. 
At each FS iteration, collocation points are uniformly resampled: $N_{pde} = 5000$, $N_{ic} = 2000$, $N_{intf} = 1000$, and $N_{bc} = 500$ per boundary, where $N_{intf}$ clusters points at $y=0.5$ to resolve the weak discontinuity. 
Loss weights are uniformly set to $1.0$ for PDE residuals, while the boundary, interface, and initial terms are penalized with weights $100.0$, $100.0$, and $150.0$, respectively, to enforce constraints during early transient stages.

While the preceding homogeneous examples primarily contend with boundary layers or initial temporal singularities, the layered Terzaghi's problem introduces a persistent internal spatial weak discontinuity.
As the optimization approaches the numerical noise floor, the sequential interactions between the fluid and solid networks typically excite late-stage micro-oscillations across the material interface.
Consequently, blindly extracting the model at the final prescribed iteration may yield a suboptimal state corrupted by these local oscillations.
To rigorously identify the optimal state for such highly heterogeneous systems without relying on exact analytical solutions, we propose an unsupervised physical validation strategy.
Specifically, we generate a dense, deterministic spatio-temporal validation grid that explicitly avoids the $t=0$ singularity to prevent unbounded time-derivatives during evaluation.
At the end of each FS iteration, we evaluate the monolithic, fully-coupled mass conservation residual of the original Biot's PDE (excluding the artificial stabilization term $\beta_{FS}$) on this validation grid.

To ensure sufficient convergence for this heterogeneous problem, the maximum number of FS iterations was increased to 100.
During the training process, we observed that while the macroscopic fields reach a satisfactory global equilibrium around the 50th iteration, resolving the steep local gradients near the material interface at early transient stages demands further optimization.
Standard evaluation metrics or visual inspections might prematurely halt the training at this pseudo-converged state.
Designed specifically for heterogeneous media where interfacial errors dominate, our validation strategy effectively tracks this micro-refinement process. The minimum physical validation loss, which perfectly coincides with the global minimum of the true relative $L^2$ error, was robustly identified at the 96th iteration.

Figure \ref{fig:layered_terzaghi} presents the 1D centerline profile comparison at various snapshots ($t=0.01, 0.1$, and $1.0$). Furthermore, Figure \ref{fig:layered_terzaghi_contour} illustrates the 2D spatial contours of the predicted fields and their absolute errors at $t=1.00$, confirming high-fidelity convergence across the entire heterogeneous domain.

\begin{figure}[!htb]
\centering
\includegraphics[width=\linewidth]{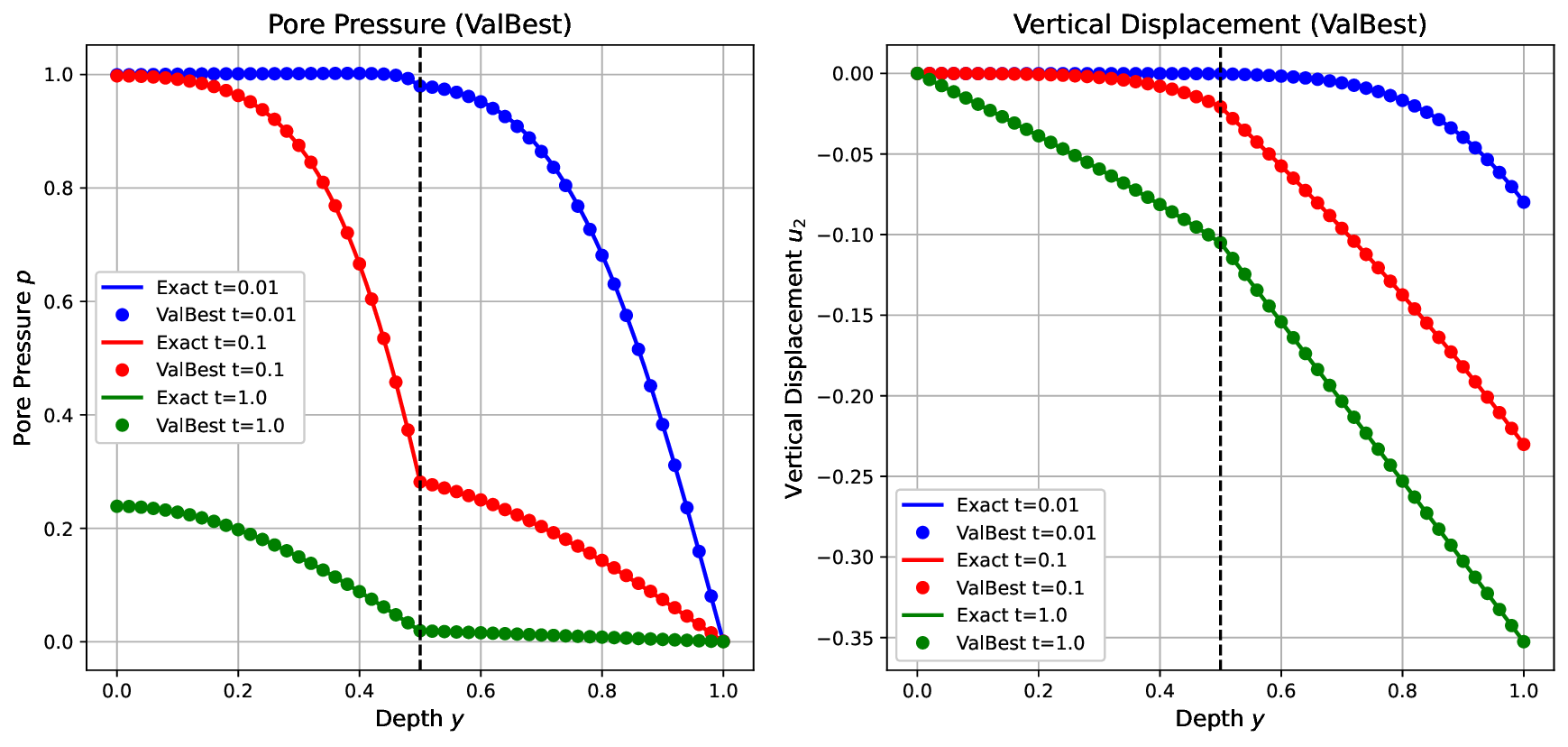}
\caption{Comparison of the physically-validated predictions (FS-ENGD-PINN) and exact solutions for the layered Terzaghi's problem along the centerline $x=0.5$. The left panel shows the pore pressure $p$, and the right panel displays the vertical displacement $u_2$ at selected time snapshots ($t=0.01, 0.1, \text{ and } 1.0$). The vertical dashed line at $y=0.5$ indicates the material interface.}
\label{fig:layered_terzaghi}
\end{figure}

\begin{figure}[!htb]
\centering
\includegraphics[width=\linewidth]{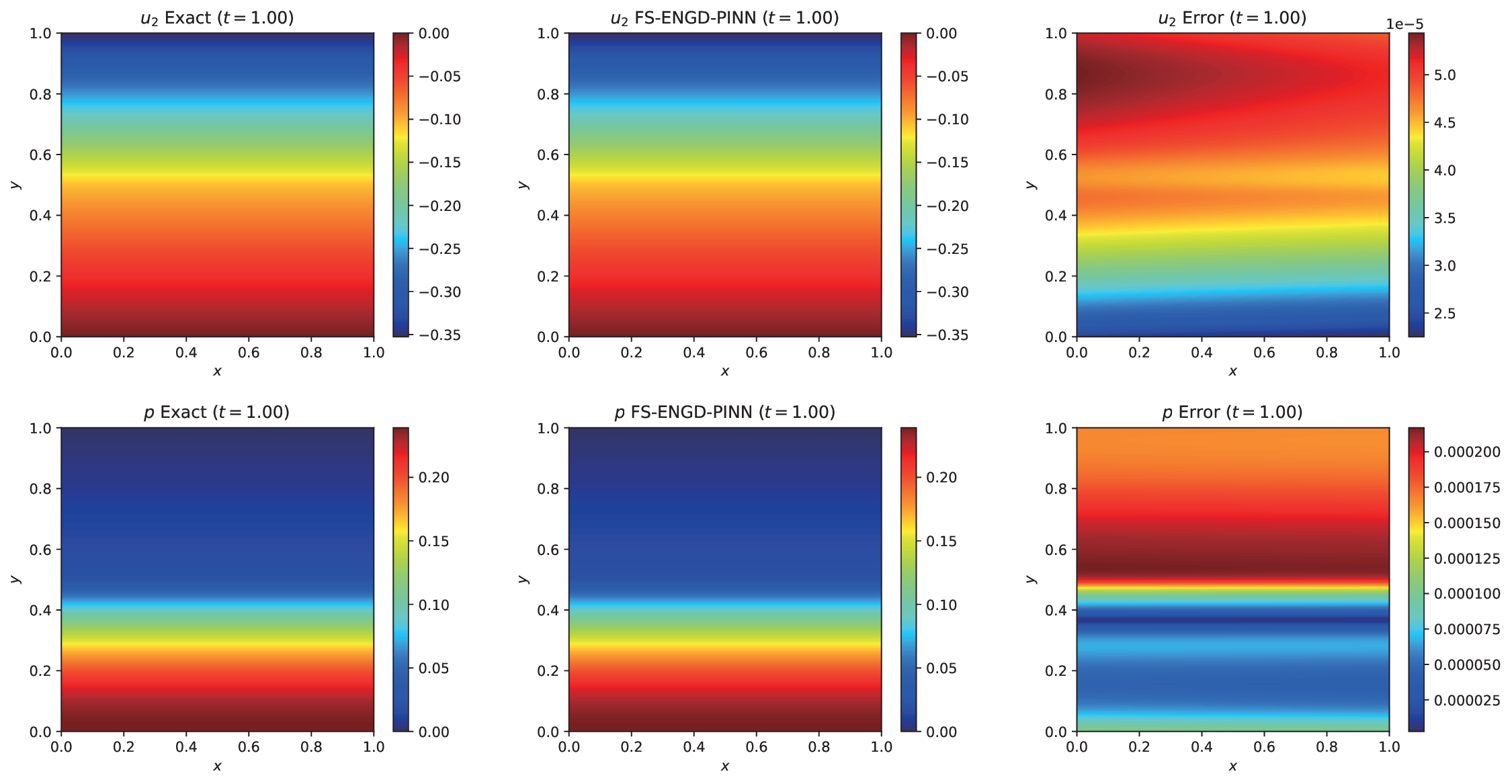}
\caption{Comparison of the predicted and exact solutions for Example 5.}
\label{fig:layered_terzaghi_contour}
\end{figure}

It is noteworthy to highlight the network's prediction of the horizontal displacement, $u_1$, which theoretically vanishes ($u_1 \equiv 0$) under uniaxial strain. 
Despite training within an unconstrained 2D domain, the FS-ENGD-PINN framework successfully precludes spurious lateral deformations, suppressing the absolute $L^2$ error of $u_1$ to $1.17 \times 10^{-6}$. 
This confirms that the solver rigorously respects the symmetric roller conditions and eliminates parasitic cross-contamination between displacement components.

\subsection{Performance near the Incompressible Limit (Three-field Model)}

In this subsection, we investigate the robustness of the proposed FS-ENGD-PINN framework for the three-field Biot's model as the material approaches the incompressible limit.
We adopt the geometric and source setup from Example 3 but vary the Poisson's ratio $\nu$ from $0.3$ to $0.4999999$.
Accordingly, the first Lam\'{e} parameter $\lambda$ increases from $O(1)$ to $O(10^6)$, introducing severe numerical stiffness and potential volumetric locking.

Table \ref{tab:incompressible_limit} summarizes the quantitative evaluation metrics for $\nu=0.4999$ and $\nu=0.4999999$. 
We explicitly compared the standard two-field formulation with our proposed three-field formulation.
For the two-field model, even when the number of FS iterations is increased to 50, it still shows obvious signs of volumetric locking. 
As shown in Table \ref{tab:incompressible_limit}, when $\nu=0.4999$, the two-field model started to have problems, and the relative $L^2$ errors for displacement reached the order of $10^{-1}$. 
More critically, at the limit value of $\nu=0.4999999$, the two-field model completely failed to capture the mechanical deformation, and the relative $L^2$ errors for displacement exceeded $1.0$, although the pressure field remained accurate.

\begin{table}[!htb]
    \centering
    \caption{Comparison of evaluation metrics between Two-field and Three-field models near the incompressible limit ($\nu \to 0.5$).}
    \label{tab:incompressible_limit}
    \renewcommand{\arraystretch}{1.2}
    \resizebox{\linewidth}{!}{
    \begin{tabular}{llcccc}
        \toprule
        $\nu$ & Model & Metric & $p$ & $u_1$ & $u_2$ \\
        \midrule
        \multirow{4}{*}{$0.4999$} 
        & \multirow{2}{*}{Two-field} 
          & Rel$L^2$ Error & 1.0168e-04 & 7.7842e-02 & 7.4665e-02 \\ 
          & & $H^1$ Error   & 1.3249e-04 & 4.8576e-01 & 3.5782e-01 \\
        \cmidrule{2-6}
        & \multirow{2}{*}{Three-field} 
          & Rel$L^2$ Error & 1.9694e-05 & 5.3202e-04 & 2.9748e-04 \\
          & & $H^1$ Error   & 6.3255e-05 & 1.5869e-03 & 1.2325e-03 \\
        \midrule
        \multirow{4}{*}{$0.4999999$} 
        & \multirow{2}{*}{Two-field} 
          & Rel$L^2$ Error & 1.1944e-05 & 1.0506e+00 & 9.9704e-01 \\ 
          & & $H^1$ Error   & 2.2414e-05 & 4.2156e+00 & 4.2051e+00 \\ 
        \cmidrule{2-6}
        & \multirow{2}{*}{Three-field} 
          & Rel$L^2$ Error & 3.2166e-05 & 2.9726e-04 & 1.7170e-04 \\
          & & $H^1$ Error   & 4.8186e-05 & 1.3911e-03 & 1.0159e-03 \\
        \bottomrule
    \end{tabular}
    }
\end{table}

To better understand this failure mode, we plot the convergence history of the component-wise errors for the two-field model in Figure \ref{fig:locking_convergence}.
Figure \ref{fig:locking_convergence} (a) shows that for $\nu=0.4999$, the displacement optimization is extremely sluggish, requiring significantly more FS iterations to reduce the error.
More alarmingly, Figure \ref{fig:locking_convergence} (b) shows that when $\nu=0.4999999$, although the pressure error (grey dashed line) converges rapidly, the displacement errors (solid line and crossed line) stagnate at a relative error of $1.0$ throughout the entire iteration process. This indicates that the strong constraint $\nabla\cdot\boldsymbol{u}\approx \boldsymbol{0}$ causes the optimizer to fall into a poor local minimum point. At this point, the neural network predicts a trivial zero displacement field to satisfy incompressibility, leading to a complete failure in capturing the mechanical response.

\begin{figure}[!htb]
    \centering
    \begin{minipage}[b]{0.48\textwidth}
        \centering
        \includegraphics[width=\linewidth]{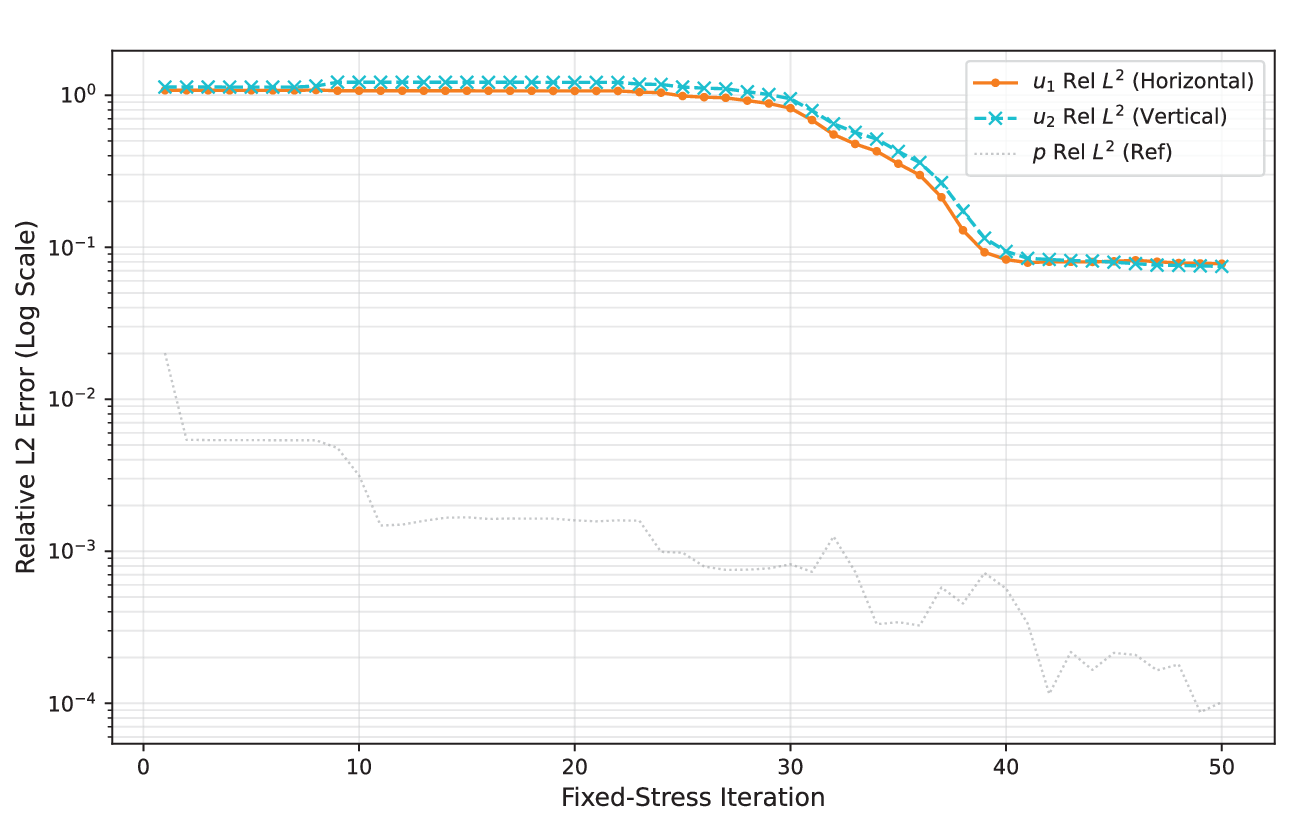}
        \centerline{(a) $\nu = 0.4999$}
    \end{minipage}
    \hfill
    \begin{minipage}[b]{0.48\textwidth}
        \centering
        \includegraphics[width=\linewidth]{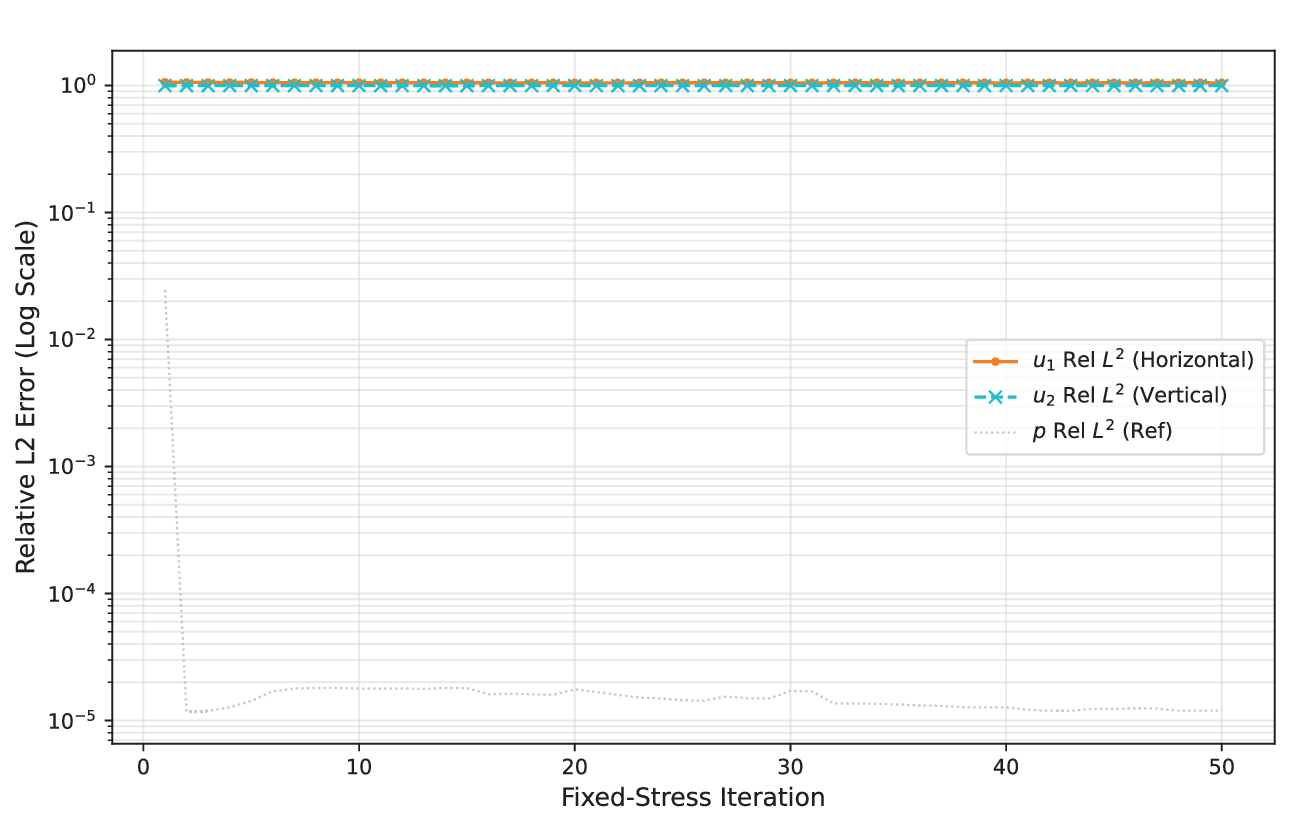}
        \centerline{(b) $\nu = 0.4999999$}
    \end{minipage}
    \caption{Convergence history of the relative $L^2$ errors for the \textbf{Two-field model} over 50 Fixed-Stress iterations. (a) At $\nu=0.4999$, displacement errors decrease slowly. (b) At $\nu=0.4999999$, displacement errors stagnate at $\approx 1.0$, indicating severe volumetric locking, while pressure errors converge independently.}
    \label{fig:locking_convergence}
\end{figure}

In stark contrast, the three-field formulation maintains exceptional accuracy and stability regardless of the Poisson's ratio.
As detailed in Table \ref{tab:incompressible_limit}, even at the extreme limit of $\nu=0.4999999$, the three-field model achieves relative $L^2$ errors for displacement on the order of $10^{-4}$, and $H^1$ errors on the order of $10^{-3}$.
This confirms that introducing the total pressure $\xi$ as an independent variable effectively alleviates volumetric locking, enabling accurate resolution of near-incompressible poroelastic problems.
Figure \ref{fig:example7} visually compares the predicted and exact solutions for $\nu=0.4999999$ of the three-field model, demonstrating excellent agreement.

\begin{figure}[!htb]
\begin{center}
\includegraphics[width=\linewidth]{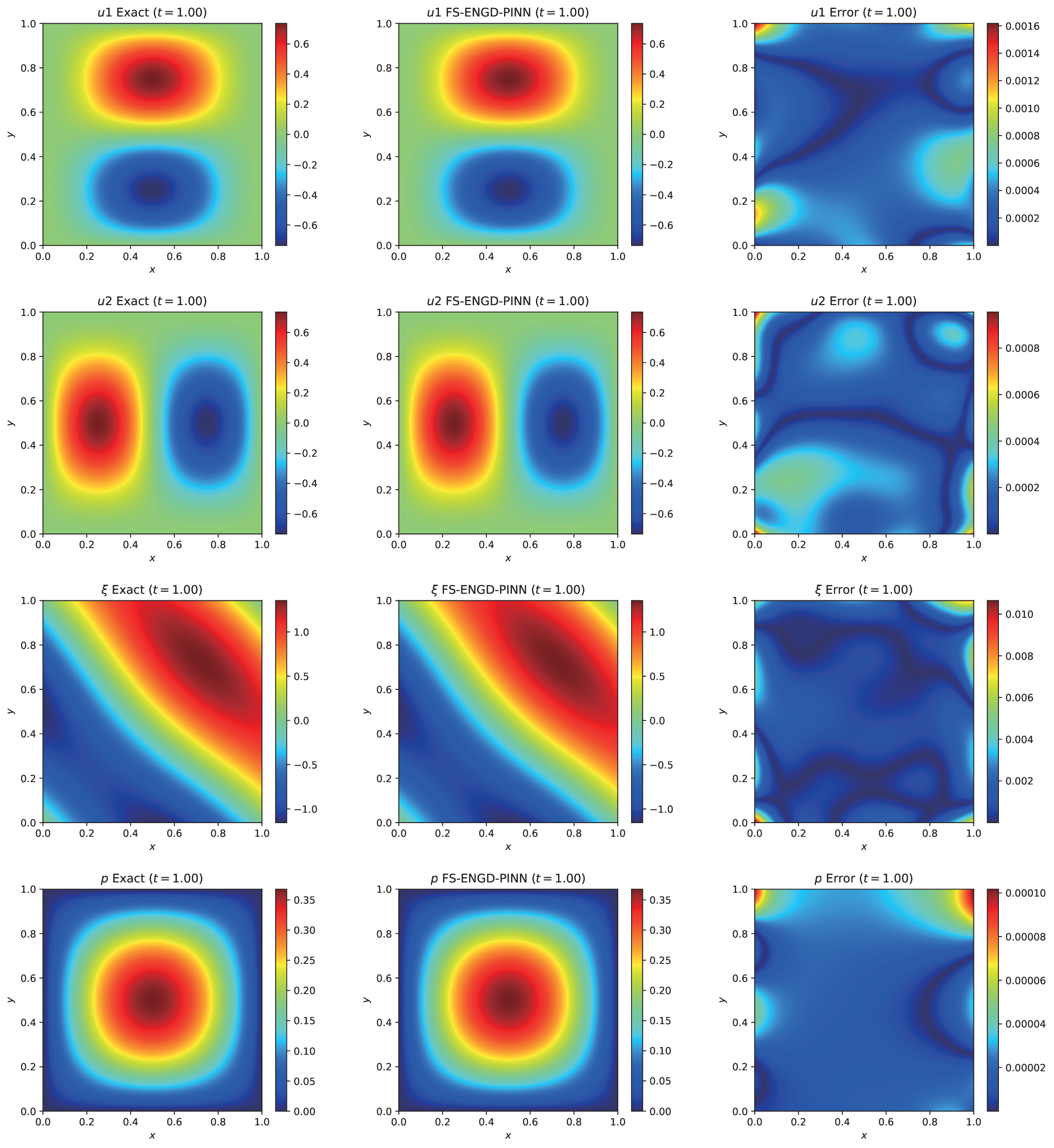}
\end{center}
\caption{Comparison of the predicted and exact solutions for the three-field model with $\nu=0.4999999$.}\label{fig:example7}
\end{figure}

\section{Conclusion and Future Work}\label{sec:conclusion}

In this work, we proposed FS-ENGD-PINN, a decoupled learning-based framework for Biot's poroelasticity that integrates the Fixed-Stress splitting scheme with Energy Natural Gradient Descent (ENGD). 
This two-level approach systematically resolves the numerical stiffness stemming from physical multi-physics coupling and optimization ill-conditioning. 
Numerical benchmarks demonstrate the framework's capability to resolve sharp initial temporal singularities, the non-monotonic Mandel-Cryer effect, and persistent internal weak discontinuities in layered heterogeneous media, while maintaining scalability in three dimensions. 
Furthermore, the three-field mixed formulation eliminates spurious pressure oscillations and successfully precludes volumetric locking near the incompressible limit.

Future research will focus on the following directions:
\begin{enumerate}
    \item Extending the decoupled framework to simulate field-scale heterogeneous reservoirs and coupled fluid-driven fracture propagation;
    \item Investigating nonlinear poromechanics involving finite elastic deformations or deformation-dependent permeability tensors;
    \item Solving inverse problems to robustly identify subsurface material parameters from sparse or noisy observational data;
    \item Integrating neural operators to enhance generalization across varying physical parameter ranges and domain geometries.
\end{enumerate}

\section*{Acknowledgments}
The work of KS and MF was supported in part by the National Natural Science Foundation of China (Grant No. 11971337). 
The work of MC was supported in part by the I-GAP Small Tech Transfer Grant
I-GAP OTT-STT-272, sponsored by the Office of Technology Transfer at Morgan State
University, and the support of The University of Maryland Baltimore Life Science
Discovery (UM-BILD) Accelerator and REACH Hub Award 1U01GM152511-03.


\bibliographystyle{unsrt}
\bibliography{cicp}

\end{document}